\documentclass[letterpaper, reqno, final]{amsart}

\pdfoutput=1 

\usepackage{amsmath, amssymb, amsthm}
\usepackage{dsfont}
\usepackage[hidelinks]{hyperref}

\usepackage{graphicx}
\usepackage{cite}
\usepackage{enumerate}
\usepackage{bookmark}
\usepackage{tikz-cd} 
\usepackage{subcaption}
\usepackage{mathtools}
\usepackage{cleveref}
\usepackage{tabularx}

\crefformat{equation}{(#2#1#3)}
\crefrangeformat{equation}{(#3#1#4)--(#5#2#6)}
\crefmultiformat{equation}{(#2#1#3)}%
{ and~(#2#1#3)}{, (#2#1#3)}{ and~(#2#1#3)}

\crefrangeformat{thm}{Theorems #3#1#4 to #5#2#6}
\crefmultiformat{thm}{Theorems #2#1#3}%
{ and~#2#1#3}{, #2#1#3}{ and~#2#1#3}

\usepackage[activate={true,nocompatibility},final,tracking=true,kerning=true,spacing=true,factor=1100,stretch=20,shrink=20]{microtype}
\microtypecontext{spacing=nonfrench}

\usepackage[right]{showlabels}

\DeclareMathOperator{\sgn}{sgn}

\DeclareMathOperator*{\slim}{s-lim}
\DeclareMathOperator*{\wlim}{w-lim}

\newcommand{\jBra}[1]{\langle #1 \rangle}

\newcommand{\bbR}{\mathbb{R}}

\newcommand{\bbC}{\mathbb{C}}
\newcommand{\bbOne}{\mathds{1}}

\newcommand{\cF}{\mathcal{F}}

\newcommand{\uwb}{u_\textup{wb}}
\newcommand{\uloc}{u_\textup{loc}}
\newcommand{\ulow}{u_\textup{low}}
\newcommand{\urem}{u_\textup{rem}}
\newcommand{\ulin}{u_\textup{lin}}
\newcommand{\Jfree}{J_\textup{free}}
\newcommand{\tJfree}{\tilde{J}_\textup{free}}

\newcommand{\Pin}{P^\textup{in}}
\newcommand{\Pout}{P^\textup{out}}
\newcommand{\Plow}{P^\textup{low}}

\newcommand{\cPin}{\mathcal{P}^\textup{in}}
\newcommand{\cPout}{\mathcal{P}^\textup{out}}

\newcommand{\Djfwd}{\mathcal{D}^{(j, \textup{fwd})}}
\newcommand{\Dzfwd}{\mathcal{D}^{(0, \textup{fwd})}}
\newcommand{\Djbwd}{\mathcal{D}^{(j, \textup{bwd})}}

\newtheorem{thm}{Theorem}

\newtheorem{prop}[thm]{Proposition}
\newtheorem{lemma}[thm]{Lemma}
\newtheorem{cor}[thm]{Corollary}
\newtheorem{rmk}{Remark}

\begin{document}
\title[Weakly localized states have localized energy]{Weakly localized states of one dimensional Schr\"odinger equations have localized energy}
\author[Avy Soffer]{Avy Soffer}
 \address[Avy Soffer]{\newline
        Department of Mathematics, \newline
         Rutgers University, New Brunswick, NJ 08903 USA.}
  \email[]{soffer@math.rutgers.edu}
\author[Gavin Stewart]{Gavin Stewart}
 \address[Gavin Stewart]{\newline
        Department of Mathematics and Computer Science, \newline
         Queensborough Community College, Bayside, NY 11364 USA.}
  \email[]{gavin.stewart@qcc.cuny.edu}
\thanks{
A.Soffer is supported in part by NSF-DMS Grant number 2205931.  G. Stewart's work is partially supported by an AMS-Simons travel grant.  Parts of this work were completed while G. Stewart was a postdoc at Rutgers University and at Arizona State University.
}
\date{\today}
\begin{abstract}
    We study the asymptotics of the Schr\"odinger equation with time-dependent potential in dimension one.  Assuming that the potential decays sufficiently rapidly as $|x| \to \infty$, we prove that the solution can be written as the sum of a free wave $e^{-it\Delta} u_+$ and a weakly bound component $u_{\textup{wb}}(t)$.  Moreover, we show that the weakly bound part decomposes as $u_{\textup{wb}}(t) = u_{\textup{loc}}(t) + o_{\dot{H}^1}(1)$, where $\partial_x u_\textup{loc}(t)$ is localized near the origin uniformly in time.  Since decay conditions on the potential do not preclude resonances unless $d \geq 5$, our results can be seen as a natural extension of~\cite{taoConcentrationCompactAttractor2007,soffer2023soliton} to the lower-dimensional case.
\end{abstract}

\maketitle

\section{Introduction}

We will study the equation
\begin{equation}\label{eqn:main-eqn}
    i\partial_t u - \Delta u + V(x,t)u = 0
\end{equation}
where the potential $V$ is assumed to satisfy the conditions
\begin{equation}\label{eqn:V-decay}
    \lVert \jBra{x}^\sigma V \rVert_{L^\infty_{t,x}} < \infty
\end{equation}
and
\begin{equation}\label{eqn:dV-decay}
    \lVert \jBra{x}^{\sigma+1} \partial_x V \rVert_{L^\infty_{t,x}} < \infty
\end{equation}
for some $\sigma > 2$.  

Based on the localization of the potential, it would be natural to expect that solutions $u$ to~\eqref{eqn:main-eqn} decompose asymptotically into a localized bound state component $\uloc$ together with linearly scattering radiation.  Currently, such a precise characterization of solutions is only known in dimensions $d \geq 5$~\cite{taoConcentrationCompactAttractor2007,taoGlobalCompactAttractor2008b,soffer2023soliton}.  In lower dimensions, the best we can say is that $u$ decomposes asymptotically into linearly scattering radiation and a slowly spreading \textit{weakly bound} state $\uwb$:
\begin{equation*}
    u(x,t) = \uwb(x,t) + e^{-it\Delta}u_+ + o_{H^1}(1)
\end{equation*}
where $\uwb(t)$ becomes orthogonal to any solution to the free Schr\"odinger equation as $t \to \infty$ and is localized to the region $|x| < t^{1/2+}$:
\begin{equation*}
    \lim_{t \to \infty}\lVert \uwb \rVert_{L^2(|x| \geq t^{1/2+})} = 0
\end{equation*}
see~\cite{soffer2022large,sofferScatteringLocalizedStates2024a,liu2025large,soffer2024new}.  The significance of the threshold at dimension $5$ was explained in~\cite{taoConcentrationCompactAttractor2007} in terms of the applicability of the double Duhamel trick, which can also be used to prove the nonexistence of resonances in dimension $d \geq 5$ for (time-independent) potentials satisfying~\eqref{eqn:V-decay}.  

Since resonances lie in $\dot{H}^1$ but not in $L^2$, it is natural to expect that we can prove localization (and not just slow spreading) for $\partial_x \uwb$.  Our main result in~\Cref{thm:main-thm} shows that this is indeed the case: The weakly bound component $\uwb$ decomposes as $\uwb(x,t) = \uloc(x,t) + o_{\dot{H}^1}(1)$, where $\jBra{x}^\theta \partial_x \uloc \in L^\infty_t L^2_x$ for any $\theta \in (0,1)$.  A similar result was proved by Tao in~\cite{taoAsymptoticBehaviorLarge2004b} for the radial cubic NLS in dimension $3$.

\subsection{Previous work}

In the case where $V$ is time-independent, the asymptotics for~\eqref{eqn:main-eqn} are by now quite well understood, even if we significantly relax the decay assumptions~\cref{eqn:V-decay,eqn:dV-decay}: See~\cite{derezinskiScatteringTheoryClassical1997,yafaevMathematicalScatteringTheory2010,reedMethodsModernMathematical2003,naumkinMatrixNonlinearSchrodinger2023} and the references therein.  In this case, the spectral theorem allows us to decompose solutions into two parts: A discrete part, which can be written as a time quasiperiodic sum of eigenfunctions of $-\Delta + V$; and a continuous part, which under the assumption~\eqref{eqn:V-decay} with $\sigma > 1$ approaches (quantum) free flow as $t \to \infty$.  In particular, the long-time asymptotics are given by
\begin{equation}\label{eqn:u-decomp-intro}
    u(x,t) = \sum e^{-itE_n} \psi_n(x) + e^{-it\Delta} u_+ + o_{{H^1}}(1)
\end{equation}
where each $\psi_n$ is an eigenfunction of $-\Delta + V$ with eigenvalue $-E_n$, the sum runs over the discrete spectrum of $-\Delta + V$, and convergence is understood in the $H^1$ sense.  A similar decomposition also holds for coupled systems of Schr\"odinger equations describing the interaction of multiple particles.  There, we must consider multiple scattering channels (corresponding to different particles becoming bound) with the conclusion that solutions asymptotically arrange themselves into bound states and free particles which become widely separated as $t \to \infty$~\cite{sigal1987n,enss1983asymptotic,hunziker2000time,derezinskiScatteringTheoryClassical1997}.  

Much less is understood in the case where $V$ is allowed to vary in time.  If the time dependence is small, it is often possible to recover the same results as in the time-independent case.  For example, for a small, time-dependent potential in dimension $3$, it is possible to prove global Strichartz estimates, which then imply scattering~\cite{rodnianskiTimeDecaySolutions2004}.  There are also a number of results on nonlinear Schr\"odinger equations with linear potentials~\cite{germainNonlinearResonancesPotential2015,germainNonlinearSchrodingerEquation2018,naumkinNonlinearSchrodingerEquations2018,chen1dimensionalNonlinearSchrodinger2022,stewartAsymptoticsCubic1D2024,frohlichSolitaryWaveDynamics2004,bronskiSolitonDynamicsPotential2000,holmerSolitonInteractionSlowly2008,sofferSelectionGroundState2004,soffer2005theory,chen1dCubicNLS2024,delortModifiedScatteringOdd2016,naumkinSharpAsymptoticBehavior2016,chenLongtimeDynamicsSmall2023a,pusateriBilinearEstimatesPresence2024a,leger3DQuadraticNLS2020,legerGlobalExistenceScattering2021}, including notable results on the asymptotic stability of solitons~\cite{chenLongtimeDynamicsSmall2023a,collotAsymptoticStabilitySolitary2023,chenAsymptoticStabilityMultisolitons2025}; the interested reader can find a review in~\cite{germainReviewAsymptoticStability2024}. See also \cite{gang2006relaxation,gang2008asymptotic,gang2006soliton}

When $V$ has a large time-dependent part, it has not yet been possible to prove results as precise as~\eqref{eqn:u-decomp-intro}.  Currently, the most complete results are those of~\cite{taoGlobalCompactAttractor2008b,soffer2023soliton}, where solutions to radial nonlinear Schr\"odinger-type equations in dimensions $d \geq 5$ were shown to decompose as
\begin{equation}\label{eqn:u-decomp-intro-ii}
    u(x,t) = \uloc(x,t) + e^{-it\Delta}u_+ + o_{H^1}(1)
\end{equation}
for some function $\uloc$ with $\jBra{x}^\delta \uloc \in L^\infty_tL^2_x$ for some constant $\delta$ depending on the dimension.   In the non-radial case,~\cite{taoConcentrationCompactAttractor2007} proved that solutions decompose as
\begin{equation*}
    u(x,t) = \sum \uloc^{(j)}(x-x_j(t),t) + e^{-it\Delta}u_+ + o_{H^1}(1)
\end{equation*}
for some localized functions $\uloc^{(j)}$ and $x_j(t)$ such that $|x_i(t) - x_j(t)| \to \infty$ ($i\neq j$).  Here, localization means $\lim_{R \to \infty} \sup_t \lVert \uloc^{(j)}(x,t) \rVert_{L^2(|x| \geq R)} = 0$.

In lower dimensions, it has not yet been possible to prove such a precise decomposition as~\eqref{eqn:u-decomp-intro-ii}.  However, we do have the weaker decomposition
\begin{equation*}
    u(x,t) = \uwb(x,t) + e^{-it\Delta}u_+ + o_{H^1}(1)
\end{equation*}
where $\uwb(t)$ is concentrated in the region $|x| \lesssim t^{1/2+}$~\cite{soffer2022large,liu2025large,sofferScatteringLocalizedStates2024a}.  Moreover, in dimension $3$ for the radial cubic nonlinear Schr\"odinger equation, Tao proved~\cite{taoAsymptoticBehaviorLarge2004b} that we can further decompose $\uwb$ as $\uwb(t) = \uloc(t) + o_{\dot{H}^1}(1)$, where
\begin{equation}\label{eqn:tao-u-decomp}
    |\nabla_x^j \uloc(x,t)| \lesssim \jBra{x}^{-3/2 - j +}
\end{equation}

\subsection{Main result and sketch of the proof}

The main result in our paper states that solutions of~\eqref{eqn:main-eqn} decompose into a linearly scattering radiation term and a term with localized kinetic energy:

\begin{thm}\label{thm:main-thm}
    Suppose $u$ solves~\eqref{eqn:main-eqn} with a potential $V$ satisfying 
    \begin{equation}\label{eqn:V-bds}
    \sup_t |V(x,t)| + |\jBra{x} \partial_x V(x,t)| \lesssim \jBra{x}^{-\sigma}
    \end{equation}
    for some $\sigma > 2$, that $\sup_t\lVert u(t) \rVert_{H^1} < \infty$, and that $\theta \in \left(0, 1\right)$ is a constant.  Then, we have the decomposition
    \begin{equation}\label{eqn:main-thm-decomp}
        u(t) = e^{-it\Delta} u_+ + \uloc(t) + o_{\dot{H}^1}(1)
    \end{equation}
    where $u_+ \in H^1$ and $\uloc$ has a localized derivative in the sense that
    \begin{equation}\label{eqn:uloc-desired}
        \jBra{x}^\theta \partial_x\uloc \in L^\infty_tL^2_x
    \end{equation}
    Moreover, if $V$ satisfies the symbol-type bounds
    \begin{equation}\label{eqn:V-symb-bds}
        \sup_{k \in \{1,2,\cdots,n\}}\sup_t |\jBra{x}^k\partial_x^k V(x,t)| \lesssim \jBra{x}^{-\sigma}
    \end{equation}
    for $n > 1$ and $\theta \in \left(0, 1\right)$ then we can redefine $\uloc$ in~\eqref{eqn:main-thm-decomp} such that
    \begin{equation}\label{eqn:uloc-higher-order}
        \jBra{x}^{\theta k}\partial_x^{k} \uloc \in L^\infty_t L^2_x
    \end{equation}
    for $1 \leq k \leq n$.
\end{thm}

\begin{rmk}
    It is not clear whether $\uloc$ is also localized in $L^2$ or not.  We note that~\cite{soffer2022existence} showed that in high dimensions, it is possible to adiabatically rescale a potential $V$ that admits a bound state in such a way that $\uwb$ (and hence $\uloc$) spreads at rate $t^{1/2-}$, nearly saturating the upper bound of $t^{1/2+}$.  In particular, these self-similar bound states satisfy both~\eqref{eqn:uloc-desired} and~\eqref{eqn:uloc-higher-order}.
\end{rmk}

\begin{rmk}
    Compared to~\cite{soffer2022large},~\Cref{thm:main-thm} has stronger control on the nonscattering part of $u$ -- \cite{soffer2022large} only proves slow spreading of the mass of $\uwb$, while our result shows that energy of $\uwb$ is in fact localized (up to terms that decay to $0$).  The main new mechanism is a careful incoming/outgoing decomposition, which allows us to gain a derivative and fractional space weight when integrating in the favorable time direction.  This is similar to the argument in~\cite{taoAsymptoticBehaviorLarge2004b}; however, we believe our proof of the higher order bound~\eqref{eqn:uloc-higher-order} is technically simpler.
    \end{rmk}
    
    \begin{rmk}
        It seems much more difficult to prove that the mass of $\uwb$ is localized.  The arguments in~\cite{taoConcentrationCompactAttractor2007,taoGlobalCompactAttractor2008b,soffer2023soliton} depend essentially on the stronger dispersion in dimensions at least $5$, so a new argument would likely be necessary.  In the other direction, we expect that the argument we present here could be adapted to higher dimensions: The main requirement would be adapt the incoming and outgoing projections to account for the fact that there are more than two directions for waves to scatter.
\end{rmk}

\begin{rmk}
    We expect that it should be possible to prove that $|x|\partial_x \uloc$ is in $L^2$ (which corresponds to the excluded $\theta = 1$ endpoint in~\Cref{thm:main-thm}), but this seems technically out of reach.  In the other direction, if it is possible to adapt the construction self-similar solutions in~\cite{soffer2022existence} to the 1 dimensional case, then this shows that $|x|^{1+}\partial_x \uloc$ cannot be bounded, which would make the $\theta = 1$ endpoint sharp.
\end{rmk}

To prove~\Cref{thm:main-thm}, we first extract the linear radiation.  To do this, we define the \textit{free channel projection operator} $\Jfree(t)$ as
\begin{equation*}
    \Jfree(t) = e^{-it\Delta} F_{\leq t^\alpha}(|x|) e^{it\Delta} F_{\geq t^{-\delta}}(|D|)
\end{equation*}
Since solutions to the free Schr\"odinger equation concentrate on the phase-space support of $\Jfree$, we define
\begin{equation*}
    u_+ = \Omega_\textup{free}u_0 = \slim_{t \to \infty} e^{it\Delta} \Jfree(t) u(t)
\end{equation*}
Based on the decay of $V$, we can use the techniques of~\cite{soffer2022large,sofferScatteringLocalizedStates2024a} to prove that the above limit exists in $H^1$, allowing us to write
\begin{equation*}
    u(t) = \uwb(t) + e^{-it\Delta} u_+ + o_{H^1}(1).
\end{equation*}
with $\uwb(t) = (I-\Jfree(t)) u(t)$.  

In addition to the standard Duhamel representation
\begin{equation}\label{eqn:intro-duhamel-fwd}\begin{split}
    \uwb(t) =& (I-\Jfree(t)) u(t)\\
            =& (I-\Jfree(t)) e^{-it\Delta} u_0 + i (I-\Jfree(t)) \int_0^t e^{-i(t-s)\Delta} V(x,s) u(s)\; ds
\end{split}\end{equation}
we also have the backward in time representation
\begin{equation*}
    \uwb(t) = e^{-i(t-T)\Delta} (I-\Jfree(T)) u(T) - \int_t^T e^{-it\Delta} \partial_s \left(e^{is\Delta} (I - \Jfree(s)) u(s)\right)\;ds
\end{equation*}
for any $T \geq t$.  In particular, since the adjoint of $e^{it\Delta}(I - \Jfree(t))$ converges strongly to $0$ in $H^1$, we can take the weak limit as $T \to \infty$ to obtain the representation
\begin{equation}\label{eqn:intro-duhamel-bwd}\begin{split}
    \uwb(t) =&-i\int_t^\infty e^{-i(t-s)\Delta} (I-\Jfree(s)) V(x,s) u(s)\;ds\\ &+ \int_t^\infty e^{-it\Delta} \partial_s\tJfree(s) e^{is\Delta} u(s)\;ds
\end{split}\end{equation}
where
\begin{equation}\label{eqn:tJfree-def}
    \tJfree(s) = F_{\leq s^\alpha}(|x|) F_{\geq s^{-\delta}}(|D|)
\end{equation}
and convergence of the integrals $\int_t^\infty \Phi(s)\;ds$ is understood in the weak topology:
\begin{equation}\label{eqn:weak-integral-convention}
    \int_t^\infty \Phi(s)\;ds := \wlim_{T \to \infty} \int_t^T \Phi(s)\;ds
\end{equation}

Examining the representation~\eqref{eqn:intro-duhamel-fwd}, we see at once that $(I-\Jfree(t))e^{-it\Delta} u_0 \to 0$ strongly in $H^1$.  In the backward in time representation~\eqref{eqn:intro-duhamel-bwd}, a careful analysis using propagation estimates shows that as $t \to \infty$,\footnote{If convergence of the integrals were in the \textit{strong} topology, this statement would be automatic.  However, because $\int_t^\infty \cdots \;ds$ is interpreted using~\eqref{eqn:weak-integral-convention}, this convergence must be checked by hand.} 
\begin{equation*}\left\lVert \int_t^\infty e^{-it\Delta} \partial_s\tJfree(s) e^{is\Delta} u(s)\;ds\right\rVert_{H^1} \to 0\end{equation*}
Thus, in each of the representations~\cref{eqn:intro-duhamel-bwd,eqn:intro-duhamel-fwd}, the main contribution comes from the Duhamel integrals involving $Vu$.  Because the time evolution in each of these integrals is in opposite directions, we can employ an incoming/outgoing decomposition to obtain improved decay in appropriate regions of phase space.  This idea goes back to the work of Enss~\cite{enss1983asymptotic}, and has been instrumental in the field~\cite{mourre1979link,rodnianskiChapter11Longtime2009,soffer2023soliton}.  Following Tao~\cite{taoAsymptoticBehaviorLarge2004b}, we will define our projections in terms of Fourier and physical space projections, although definitions in terms of spectral projectors of the dilation operator $A = x \cdot \nabla + \nabla \cdot x$ are also possible.

To define the incoming/outgoing decomposition and explain how it leads to improved decay, we recall that the Schr\"odinger propagator $e^{-it\Delta}$ is the quantum version of the classical flow $x \mapsto x - 2\xi t$.  Using this heuristic, we see that if we project to positive frequencies, the propagator $e^{-i(t-s)\Delta}$ will tend to move waves localized near the origin to the left (outward) for~\eqref{eqn:intro-duhamel-fwd} resulting in improved decay in the `incoming' region $x > 0$, while $e^{-i(t-s)\Delta}$ moves waves to the right (inward) for~\eqref{eqn:intro-duhamel-bwd}, which gives better decay in the `outgoing' region $x < 0$.  A similar heuristic also holds for negative frequencies, so as a first approximation, we might define
\begin{equation*}\begin{split}
    \Pin_{(\textup{naive})} =& \bbOne_{x > 0} \bbOne_{D > 0} + \bbOne_{x < 0} \bbOne_{D < 0}\\
    \Pout_{(\textup{naive})} =& \bbOne_{x < 0} \bbOne_{D > 0} + \bbOne_{x > 0} \bbOne_{D < 0}
\end{split}\end{equation*}
However, this is too na\"ive, as sharp cut-offs in frequency are not compatible with spatial localization.  To fix this, we split out the low-frequencies and instead work with incoming, outgoing, and low-frequency projections, which, roughly, are given by 
\begin{equation*}\begin{split}
    \Pin =& \bbOne_{x > 0} F_{\geq \jBra{x}^{-\theta}}(D) + \bbOne_{x < 0} F_{\geq \jBra{x}^{-\theta}}(-D)\\
    \Pout =& \bbOne_{x < 0} F_{\geq \jBra{x}^{-\theta}}(D) + \bbOne_{x > 0} F_{\geq \jBra{x}^{-\theta}}(-D)\\
    \Plow =& F_{\leq \jBra{x}^{-\theta}}(|D|)
\end{split}\end{equation*}
for $\theta \in (0,1)$.  (See~\Cref{sec:in-out} for precise definitions).  Thus, by writing
\begin{equation*}
    \uwb(t) = \Pin \uwb(t) + \Pout \uwb(t) + \Plow \uwb(t)
\end{equation*}
we can use the forward-in-time Duhamel integral~\eqref{eqn:intro-duhamel-fwd} to obtain spatial decay for the incoming term $\Pin \uwb(t)$ and the backward-in-time integral~\eqref{eqn:intro-duhamel-bwd} to get spatial decay for $\Pout \uwb(t)$.  Finally, the terms $\Plow \uwb(t)$ are localized to frequencies $|\xi| \lesssim \jBra{x}^{-1+}$, so their derivatives are manifestly localized.  This suffices to prove~\eqref{eqn:uloc-desired}.

It only remains to prove the higher-order bound~\eqref{eqn:uloc-higher-order}.  The issue here is that the incoming/outgoing projectors only allow us to gain one factor of $\jBra{x}^\theta \partial_x$:
\begin{equation*}
    \left\lVert (\jBra{x}^\theta \partial_x)^2 \Pin \int_0^t e^{-i(t-s)\Delta} f(s)\;ds \right\rVert_{L^2} \lesssim \sum_{m = 0}^{1} \lVert \jBra{x}^\sigma (\jBra{x}^\theta \partial_x)^m f(s) \rVert_{L^\infty_t H^\epsilon_x}
\end{equation*}
and similarly for the outgoing projection.  For our application, $f(s) = Vu(s)$.  If we could replace $Vu$ with $V\uloc$, we could simply proceed by induction.  Indeed, this is essentially the approach taken in~\cite{taoGlobalCompactAttractor2008b} to obtain spatial decay in $d \geq 5$.  However, this would require us to prove that $\uloc$ is a solution to~\eqref{eqn:main-eqn}, which seems out of reach in lower dimensions.  Instead, recalling the decomposition~\eqref{eqn:main-thm-decomp}, we write
\begin{equation*}
    Vu = V\uloc + V (e^{-it\Delta}u_+ + \urem)
\end{equation*}
where $\urem(t)$ vanishes in $\dot{H}^1$ as $t \to \infty$.  By leveraging the properties of the incoming and outgoing projectors and the decay of $V (e^{-it\Delta}u_+ + \urem)$ in time, we find that
\begin{equation*}
    \left\lVert \partial_x \Pin \int_0^t e^{-i(t-s)\Delta} V (e^{-is\Delta}u_+ + \urem(s))\;ds \right\rVert_{H^1} \to 0
\end{equation*}
and similarly for the outgoing projection.  Absorbing these terms into the remainder term $o_{\dot{H}^1}(1)$ thus gives~\eqref{eqn:uloc-higher-order} with $n = 2$, and further iteration gives the result for any $n$.

\section*{AI Statement}

The authors used Claude Opus 5.5 to assist with copyediting and to verify calculations and intermediate results. All mathematical statements, proofs, calculations, and references in the manuscript were independently verified by the authors.

\section{Preliminaries}
\subsection{Notation and conventions}

We use the notation $A \lesssim B$ to mean that $A \leq CB$ for some fixed implicit constant $C>0$.  If $A \lesssim B$ and $B \lesssim A$, then we write $A \sim B$.  For implicit constants that depend on parameters, we will write
$A \lesssim_P B$ to mean $A \leq C_P B$ for a constant $C_P$ which is allowed to depend on $P$.

Given a set $K \subset \bbR$, we will denote the indicator function of $K$ by $\bbOne_K(x)$.  We also define the smooth functions $F, F_\leq, F_\geq: \bbR \to [0,1]$ such that $F$ is supported on $[1/2, 2]$ with
$$\sum_{j =-\infty}^\infty F\left(\frac{r}{2^j}\right) = \bbOne_{\bbR^+}(r)$$
and
\begin{equation*}\begin{split}
    F_{\leq}(r) =& \sum_{j=-\infty}^{0} F\left(\frac{r}{2^j}\right)\\
    F_{\geq}(r) =& \sum_{j=0}^{\infty} F\left(\frac{r}{2^j}\right)
\end{split}
\end{equation*}
Given $C > 0$, we define the functions
\begin{equation*}\begin{split}
    F_C(r) =& F(r/C)\\
    F_{\leq C}(r) =& F_\leq(r/C)\\
    F_{< C}(r) =& F_{\leq C/2}(r)\\
    F_{\lesssim C}(r) =& F_{\leq C 2^{10}}(r)\\
    F_{\ll C}(r) =& F_{< C 2^{-10}}(r)\\
    F_{\sim C}(r) =& F_{\lesssim C}(x) - F_{\ll C}(r)\\
    F_{\geq C}(r) =& F_\geq(r/C)\\
    F_{> C}(r) =& F_{\geq 2C}(r)\\
    F_{\gtrsim C}(r) =& F_{\geq C 2^{-10}}(r)\\
    F_{\gg C}(r) =& F_{> C 2^{10}}(r)
\end{split}\end{equation*}
In particular, we have that
\begin{equation*}
    F_{< C} + F_C + F_{> C} = \bbOne_{\bbR^+} = F_{\ll C} + F_{\sim C} + F_{\gg C}
\end{equation*}
In particular, $F_{< C}(r)$ is a smooth cut-off to $0 < r \leq C/2$, $F_C(r)$ is a smooth cut-off to $C/2 \leq r \leq 2C$, and $F_{> C}(r)$ is a smooth cut-off to $r \leq 2C$.  Similarly, the functions $F_{\ll C}$, $F_{\sim C}$ and $F_{\gg C}$ are smooth cut-offs to $r \ll C$, $r \sim C$, and $r \gg C$, respectively.  We also remark that the precise constants $2^{-10}$ and $2^{10}$ are not essential for the proof, and are chosen for convenience.

To more easily track the $C$ dependence of derivatives of $F_C$, we take the convention that
\begin{equation*}
    F_C^{(a)}(x) = F^{(a)}(x/C)
\end{equation*}
(so $|F_C^{(a)}(x)|$ is bounded independent of $C$), and similarly for $F_{\leq C}$, $F_{\geq C}$, and so on.  In particular, this means that
\begin{equation*}
    \partial_x^{n} F_C(x) = C^{-n} F^{(n)}(x/C) = C^{-n} F^{(n)}_C(x)
\end{equation*}

The Fourier transform of a function $f$ is given by
\begin{equation*}
    \cF f(\xi) = \hat{f}(\xi) = \frac{1}{\sqrt{2\pi}} \int f(x) e^{-ix\xi}\;dx
\end{equation*}
As is well-known, the Fourier transform is unitary on $L^2$, with inverse
\begin{equation*}
    \cF^{-1} g(x) = \check{g}(x) = \frac{1}{\sqrt{2\pi}} \int g(\xi) e^{ix\xi}\;d\xi
\end{equation*}
Given a function $m: \bbR \to \bbR$, we define the Fourier multiplier $m(D)$ to be
\begin{equation*}
    m(D) f = \cF^{-1} m(\xi) \cF f
\end{equation*}
In particular,
\begin{equation*}
    \cF \left[m(D) f\right](\xi) = m(\xi) \hat{f}(\xi)
\end{equation*}
For our purposes, we will make frequent use of the dyadic Littlewood-Paley projectors
\begin{equation*}\begin{split}
        P_k =& F_{2^k}(|D|)\\
        P_k^\pm =& F_{2^k}(\pm D)
\end{split}
\end{equation*}
and the variants $P_{\leq C}$, $P_{\geq C}$, etc.  In particular, we have that
\begin{equation*}
    \sum_{k=-\infty}^\infty P_k f = f
\end{equation*}

\subsection{Parameters}

For ease of reference, in this subsection we give the interpretations of several parameters that recur throughout the work, and state the relationships between them in~\Cref{tab:params}.  


\begin{table}[ht]
\centering
\renewcommand{\arraystretch}{1.2}
\begin{tabularx}{\linewidth}{|>{\raggedright\arraybackslash}p{0.10\linewidth}|>{\raggedright\arraybackslash}X|>{\raggedright\arraybackslash}p{0.35\linewidth}|}
\hline
\textbf{Symbol} & \textbf{Interpretation} & \textbf{Inequalities} \\
\hline
$\alpha$ & An exponent for the spatial cut-off $F_{\leq t^\alpha}(|x|)$ appearing in $\Jfree(t)$ & $0 < \alpha < \frac{1}{2}$\\\hline
$\delta$ & An exponent for the frequency cut-off $F_{\geq t^{-\delta}}(|D|)$ appearing in $\Jfree(t)$ & $0 < \delta < \min\left(\frac{1}{2}, 1-\alpha, \alpha\right)$\\\hline
$\rho$ & Through $J(t) \sim 2^\rho$, the largest spatial scale $|x| \sim 2^\rho$ at which the dyadic projections $\Pin_j$ and $\Pout_j$ are considered before reverting the the projections $\Pin_{\geq J(t)}$ and $\Pout_{\geq J(t)}$ (respectively.) & $\rho\theta < \alpha < \rho < \frac{1}{2}$, $\sigma\rho > 1$\\\hline
$\theta$ & Used in the definition of $\Pin_j$ and $\Pout_j$, which project to $|x| \sim 2^j$ and frequencies of size $\gtrsim 2^{-\theta j}$ & $0 < \theta < 1$\\\hline
\end{tabularx}
\caption{\label{tab:params} The parameters used in this paper, together with the constraints they are required to satisfy.}
\end{table}

The various restrictions enter the proof as follows:
\begin{itemize}
    \item $\delta < \min\left(\frac{1}{2}, 1-\alpha, \alpha\right)$: Required in the nonstationary phase estimates in~\Cref{lem:space-freq-loc-decay}, which are used to prove the existence of the wave operator.
    \item $\rho > \alpha$: Required to prove time decay for $\urem^{(\geq J(t), \textup{fwd},2)}(t)$ in~\Cref{sec:energy-loc}.
    \item $\theta\rho < \alpha$: Use to ensure decay when the frequency projections in $\Pin_{\geq J(t)}$ are commuted with $t^\alpha$ when proving decay of $\urem^{(\geq J(t), \textup{bwd}, 2)}$.
    \item $\sigma \rho > 1$: Required to guarantee sufficient time decay in the region $|x| \geq t^{\beta}$ for $\frac{1}{\sigma} < \beta < \rho$ when proving that $\urem^{(\geq J(t), \textup{fwd}, 1)}$ decays.
    \item $\theta < 1$: Required by the uncertainty principle in order for the projection $\Pin_j$ and $\Pout_j$ to be (microlocally) sensible.
    \item $\rho < \frac{1}{2}$: Used to obtain time decay for $\urem^{(j,\textup{fwd})}$ by applying~\Cref{cor:full-freq-proj-supp-lem} with $\eta = \rho$, $\lambda = \theta \rho$, $\frac{1}{\sigma} < \beta < \rho$.
\end{itemize}

\subsection{Basic estimates}
\begin{lemma}\label{lem:space-freq-loc-decay}
    For $\lambda < \min(1/2, 1-\eta, 1-\beta)$ and for any positive $M$,
    \begin{equation}
        \lVert F_{\leq t^\eta}(|x|) F_{\geq t^{-\lambda}}(|D|) e^{it\Delta} F_{\leq t^\beta}(|x|) \rVert_{L^2 \to L^2} \lesssim_M t^{-M}
    \end{equation}
\end{lemma}
\begin{proof}
    We first observe that it suffices to prove the bound for $t$ large, since for small $t$ the bound follows immediately from the trivial bound
    \begin{equation*}
        \lVert F_{\leq t^\eta}(|x|) F_{\geq t^{-\lambda}}(|D|) e^{it\Delta} F_{\leq t^\beta}(|x|) \rVert_{L^2 \to L^2} \lesssim 1
    \end{equation*}
    The main idea in this argument (which will be used frequently in this work) is to take advantage of nonstationary phase on the frequency side.  To do this, we rewrite the operator in terms of integration against a kernel:
    \begin{equation*}
        F_{\leq t^\eta}(|x|) F_{\geq t^{-\lambda}}(|D|) e^{it\Delta} F_{\leq t^\beta}(|x|) \phi = \int K_t(x,y) \phi(y)\;dy
    \end{equation*}
    where
    \begin{equation*}
        K_t(x,y) = \frac{F_{\leq t^\eta}(|x|) F_{\leq t^\lambda}(|y|)}{2\pi} \int e^{-it\xi^2} e^{i\xi(x-y)} F_{\geq t^{-\lambda}}(|\xi|)\;d\xi
    \end{equation*}
    The integral has phase $\phi(\xi) = -t\xi^2 + \xi(x-y)$.  Under the assumption that ${\lambda < \min(1-\eta, 1-\beta)}$, we see that for large enough $t$,
    \begin{equation*}
        |\phi'(\xi)| = |x-y -2t\xi| \sim |t\xi| \gtrsim |t|^{1-\delta}
    \end{equation*}
    In particular, by repeatedly integrating by parts using the identity
    \begin{equation*}
        e^{-it\xi^2 + i\xi(x-y)} = \frac{i}{2t\xi - (x-y)} \partial_\xi e^{-it\xi^2 + i\xi(x-y)}
    \end{equation*}
    we see that
    \begin{equation*}
        K_t(x,y) = \sum_{a+b=N} C_{a,b} \frac{F_{\leq t^\eta}(|x|) F_{\leq t^\beta}(|y|)}{2\pi} \int \frac{t^a e^{-it\xi^2} e^{i\xi(x-y)}}{(2t\xi - (x-y))^{N+a}} t^{b\lambda} F^{(b)}_{\geq t^{-\lambda}}(|\xi|)\;d\xi
    \end{equation*}
    In particular, we calculate that
    \begin{equation*}\begin{split}
        \lVert K_t \rVert_{L^2_{x,y}} \lesssim& \lVert F_{\leq t^\eta}(|x|) \rVert_{L^2_x} \lVert F_{\leq t^\beta}(|y|) \rVert_{L^2_y} \sum_{a+b=N} \int \frac{|t|^a}{|2t\xi|^{N+a}} |t|^{b\lambda} |F^{(b)}_{\geq t^{-\lambda}}(|\xi|)|\;d\xi\\
        \lesssim_N& t^{(\eta + \lambda)/2} t^{-N(1-2\lambda)}
    \end{split}\end{equation*}
    Choosing $N$ sufficiently large depending on $M,\lambda, \eta$ and $\beta$ and recalling that 
    \begin{equation*}
        \lVert F_{\leq t^\eta}(|x|) F_{\geq t^{-\lambda}}(|D|) e^{it\Delta} F_{\leq t^\beta}(|x|) \rVert_{L^2 \to L^2} \leq \lVert K_t(x,y) \rVert_{L^2_{x,y}}
    \end{equation*}
    now gives the result.
\end{proof}

\begin{cor}\label{cor:loc-decay-weighted}
    For any $\eta$, $\beta,$ and $\lambda < \min(1/2, 1-\alpha, 1-\beta)$, if $|V(x)| \lesssim \jBra{x}^{-\sigma}$, then
    \begin{equation}\label{eqn:loc-decay-weighted}
        \lVert F_{\leq t^\eta}(|x|) F_{\geq t^{-\lambda}}(|D|) e^{it\Delta} V(x) \rVert_{L^2 \to L^2} \lesssim \jBra{t}^{-\beta\sigma}
    \end{equation}
\end{cor}
\begin{proof}
    The bound is trivial for $|t| \leq 1$, so we will focus on the case $|t| > 1$.  Let us write $V = F_{\geq t^\beta}(|x|) V + F_{< t^\beta}(|x|) V$.  The operator with $V$ restricted to $|x| < t^\beta$ decays arbitrarily fast in time by~\Cref{lem:space-freq-loc-decay}, while the decay of $V$ gives the bound
    \begin{equation*}
        |F_{\geq t^\beta}(|x|) V| \lesssim t^{-\beta\sigma}
    \end{equation*}
    which is enough to prove~\eqref{eqn:loc-decay-weighted}.
\end{proof}

\begin{rmk}
    The parameter $\beta$ does not appear in~\eqref{eqn:loc-decay-weighted}; however, the range of allowable $\beta$ is constrained by $\lambda$ through the constraint $\lambda < 1-\beta$, so in practice~\Cref{cor:loc-decay-weighted} means that for every $\epsilon > 0$,
    \begin{equation*}
        \lVert F_{\leq t^\eta}(|x|) F_{\geq t^{-\lambda}}(|D|) e^{it\Delta} V(x) \rVert_{L^2 \to L^2} \lesssim_\epsilon \jBra{t}^{-(1-\lambda -\epsilon)\sigma}
    \end{equation*}
\end{rmk}

We also have the following refinement:
\begin{lemma}\label{lem:proj-fast-decay-2}
    Let $j \geq 0$ and $k >  -j - 10$.  For $t \gg 2^{j-k}$, we have that
    \begin{equation}\label{eqn:proj-fast-decay-2}
        \lVert \bbOne_{|x| < t 2^{k-20}} \partial_x^m P_k e^{-it\Delta} F_{ \lesssim 2^j}(|x|) \rVert_{L^2 \to L^2} \lesssim_N \min\left( 2^{mk}, 2^{j/2}\frac{2^{(m+3/2-2N)k}}{t^{N-1/2}}\right)
    \end{equation}
\end{lemma}
\begin{proof}
    The bound of $2^{mk}$ follows immediately from the frequency localization, so we will focus on proving the second bound.  We observe that the integral kernel associated with $\bbOne_{|x| < t 2^{k-20}} \partial_x^m P_k e^{-it\Delta} F_{ \lesssim 2^j}(|x|)$ is
    \begin{equation*}
        K(x, y) = \frac{\bbOne_{|x| < t 2^{k-20}} F_{\lesssim 2^j}(|y|)}{2\pi} \int \xi^m  F_{2^k}(|\xi|) e^{i(t \xi^2 + (x-y)\xi)}\;d\xi
    \end{equation*}
    For $|x| < t 2^{k-20}$ and $|y| \lesssim 2^j$, the phase for the exponential in the integral is bounded away from zero, with magnitude
    $$|2t\xi + (x-y)| \sim t 2^k$$
    Repeated integration by parts now gives
    \begin{equation*}\begin{split}
        \left|\int \xi^m F_{2^k}(|\xi|) e^{i(t \xi^2 + (x-y)\xi)}\;d\xi\right| \leq& \sum_{a + b + c = N} C_{a,b,c} \int 2^{-ak} \xi^{m-c} \frac{|F_{2^k}^{(a)}(|\xi|)|}{t^N \xi^{N+b}}\;d\xi \\
        \lesssim_N& \frac{2^{(m+1 - 2N)k}}{t^N}
    \end{split}\end{equation*}
    Thus, we find that
    \begin{equation*}
        \lVert K(x,y) \rVert_{L^2_{x,y}} \lesssim_N 2^{j/2} \frac{2^{(m + 3/2 - 2N)k}}{t^{N-1/2}}
    \end{equation*}
    which implies~\eqref{eqn:proj-fast-decay-2}.
\end{proof}

\begin{lemma}\label{lem:half-freq-proj-supp-lem}
    Suppose $f(x)$ and $g(x)$ are bounded functions whose supports are separated by a distance $2^j$.  Then, for any $N \geq 0$,
    $$\lVert f(x) P_{\geq k}^\pm g(x) \rVert_{L^2 \to L^2} \lesssim_N \lVert f \rVert_{L^\infty} \lVert g \rVert_{L^\infty} 2^{-N(k+j)}$$
\end{lemma}
\begin{proof}
    To avoid cumbersome notation, we will prove the result for the projection to positive frequencies.  We have that
    \begin{equation*}
         P^+_{\geq k} = \sum_{\ell = k}^\infty P_{\ell}^+
    \end{equation*}
    Moreover, the operator $ f(x) P^+_{\ell} g(x)$ can be represented as integration against the kernel
    \begin{equation*}
        K_\ell(x,y) = \frac{f(x)g(y)}{2\pi} \int e^{i\xi (x-y)} F_{2^\ell}(\xi) \;d\xi
    \end{equation*}
    Repeated integration by parts shows that
    \begin{equation*}\begin{split}
        |K_\ell(x,y)|  \lesssim& \left|\frac{f(x)g(y)}{2\pi(x-y)^N2^{N\ell}} \int e^{i\xi (x-y)} F^{(N)}_{2^\ell}(\xi)\;d\xi \right|\\
                    \lesssim& \frac{f(x)g(y)}{2\pi|x-y|^N} 2^{(1-N)\ell}
    \end{split}\end{equation*}
    Thus, noting that $K_\ell(x,y) = 0$ unless $|x-y| > 2^j$, we see that for $N \geq 2$,
    \begin{equation*}
        \lVert K_\ell(x,y) \rVert_{L^\infty_yL^1_{x} \cap L^\infty_x L^1_y} \lesssim \lVert f \rVert_{L^\infty}\lVert g \rVert_{L^\infty} 2^{(1-N)(j+\ell)}
    \end{equation*}
    which by Schur's test implies that
    \begin{equation*}
        \lVert f(x) P^+_\ell g(x) \rVert_{L^2\to L^2} \lesssim \lVert f \rVert_{L^\infty}\lVert g \rVert_{L^\infty} 2^{(1-N)(j+\ell)}
    \end{equation*}
    Summing over $\ell \geq k$ now gives the result.
\end{proof}

By combining the positive and negative frequencies and observing that $P_{< k} = 1 - P_{\geq k}$, we easily obtain the following corollary:
\begin{cor}\label{cor:full-freq-proj-supp-lem}
    Suppose $f(x)$ and $g(x)$ are bounded functions whose supports are separated by a distance $2^j$.  Then, for any $N \geq 0$,
    $$\lVert f(x) P_{\geq k} g(x) \rVert_{L^2 \to L^2} \lesssim_N \lVert f \rVert_{L^\infty} \lVert g \rVert_{L^\infty} 2^{-N(k+j)}$$
    and
    $$\lVert f(x) P_{< k} g(x) \rVert_{L^2 \to L^2} \lesssim_N \lVert f \rVert_{L^\infty} \lVert g \rVert_{L^\infty} 2^{-N(k+j)}$$
\end{cor} 


We will also need the following formula, which allows us to symmetrize certain sums involving operators:
\begin{lemma}
    Given operators $A,B,$ and $C$ with $[A,C] = 0$, we have that
    \begin{equation}\label{eqn:symm-formula}
        A^2BC^2 + C^2 B A^2 = 2 (AC)B(AC) + R(A,B,C)
    \end{equation}
    where
    \begin{equation*}
        R(A,B,C) = A[[A,B],C]C + C[[C,B],A]A + A[C,[C,B]]A + C[A,[A,B]]C
    \end{equation*}
\end{lemma}

In order to deal with commutators involving space weights, the following elementary estimate will be helpful:
\begin{lemma}\label{lem:elem-ineq}
    For any $a > 0$ and $\mu \in [0, \min(a, 1)]$,
    $$||x|^a - |y|^a| \lesssim_{a, \mu} (|x|+|y|)^{a - \mu} |x-y|^\mu$$
\end{lemma}
\begin{proof}
    By interpolation, it suffices to prove the cases $\mu = 0$, $\mu = a < 1$, and $\mu = 1 \leq a$.  The case $\mu = 0$ follows immediately from the fact that $x \to x^a$ is increasing, and the case $\mu = 1 \leq a$ follows from the mean value theorem.  Finally, for $\mu = a < 1$, we note that if $|x| > |y|$, then
    \begin{equation*}\begin{split}
        |x|^a - |y|^a =& a \int_0^{|x|-|y|} (z + |y|)^{a - 1}\;dz\\
                                \leq& a \int_0^{|x|-|y|} z^{a - 1}\;dz\\
                                \leq& ||x|-|y||^a
                                \leq |x-y|^a\qedhere
    \end{split}\end{equation*}
\end{proof}

\begin{cor}\label{cor:elem-ineq}
    For any $a > 0$ and $\mu \in [0, \min(a, 1)]$,
    $$|\jBra{x}^a - \jBra{y}^a| \lesssim_{a, \mu} (\jBra{x}+\jBra{y})^{a - \mu} |x-y|^\mu$$
\end{cor}
\begin{proof}
    The case $\mu = 0$ is immediate.  For $\mu \neq 0$, we observe that
    \begin{equation}\label{eqn:jBra-deriv}
        \partial_x \jBra{x}^a = a \frac{x}{\jBra{x}} \jBra{x}^{a - 1}
    \end{equation}
    Since the right-hand side is less than $\alpha\jBra{x}^{a-1}$, the mean value theorem now suffices to prove the result with $\mu = 1 \leq a$.  Finally, for $\mu = a < 1$, we compute that for $|x| > |y|$
    \begin{equation*}\begin{split}
        \jBra{x}^a - \jBra{y}^a =& \int_0^{|x|-|y|} a \frac{z+|y|}{\jBra{z+|y|}} \jBra{z+|y|}^{a - 1}\;dz\\
        \leq& \int_0^{|x|-|y|} \alpha |z+|y||^{a - 1}\;dz\\
        \leq& |x-y|^a \qedhere
    \end{split}\end{equation*}
\end{proof}

In particular, this lemma allows us to prove the following useful result:
\begin{cor}\label{cor:comm-cor}
    For $a > 0$, $\mu \in [0, \min(1, a)]$, we have that
    \begin{equation}\label{eqn:power-comm}
        \lVert F_{\lesssim 2^j}(|x|) [|x|^a, P_k] F_{2^j}(|x|) \rVert_{L^2 \to L^2} \lesssim 2^{(a - \mu)j} 2^{-\mu k}
    \end{equation}
    and
    \begin{equation}\label{eqn:jBra-comm}
        \lVert F_{\leq 2^j}(|x|) [\jBra{x}^a, P_k] F_{\leq 2^j}(|x|)\rVert_{L^2 \to L^2} \lesssim \jBra{2^{j}}^{a - \mu} 2^{-\mu k}
    \end{equation}
\end{cor}
\begin{proof}
    For~\eqref{eqn:power-comm}, we observe that $P_k$ has a kernel given by
    $$K_k(x-y) = 2^k K(2^k(x-y))$$
    where $K(z)$ is smooth and rapidly decaying.  Thus, using~\Cref{lem:elem-ineq}, we have that
    \begin{equation*}\begin{split}
        \lVert F_{\lesssim 2^j}(|x|) [|x|^a, P_k] F_{2^j}(|x|) \rVert_{L^2 \to L^2} \!\!\leq& \lVert F_{\lesssim 2^j}(|x|) F_{2^j}(|y|) \left(|x|^a - |y|^a\right) K_k(x-y) \rVert_{L^2_{x,y}}\\
        \lesssim& 2^{(a - \mu) j} \lVert 2^k |x-y|^\mu |K(2^k(x-y))| _{L^2_{x,y}}\\
        \lesssim& 2^{(a - \mu) j} 2^{-\mu k}
    \end{split}\end{equation*}
    To obtain~\eqref{eqn:jBra-comm}, we repeat the argument using~\Cref{cor:elem-ineq}.
\end{proof}

\subsection{Propagation estimates}

At several points in our argument, we will need to use the method of \textit{propagation estimates}.  Given an operator $B$ and a function $u: \bbR \times \bbR \to \bbC$, we define
\begin{equation*}
    \langle B \rangle_{t,u} := \langle Bu(t), u(t) \rangle
\end{equation*}
When the function $u$ is clear from context, we will often omit it from the notation and simply write $\langle B \rangle_t$.  We will also call the operator $B$ the \textit{propagation observable}.  If $u$ solves
the abstract Schr\"odinger equation
\begin{equation*}
    i\partial_t u = A(t) u
\end{equation*}
with $A(t)$ self-adjoint, and if $B(t)$ is a (possibly time-dependent) propagation observable, then we have the identity
\begin{equation}\label{eqn:PO-deriv}
    \frac{d}{dt} \langle B(t) \rangle_t = \langle D_H B(t) \rangle_t
\end{equation}
where
\begin{equation}\label{eqn:H-deriv-def}
    D_H B(t) := \frac{d}{dt} B(t) + i[A(t), B(t)]
\end{equation}
is the Heisenberg derivative.  Integrating gives the identity
\begin{equation*}
    \langle B(t_2) \rangle_{t_2} - \langle B(t_1)\rangle_{t_1} = \int_{t_1}^{t_2} \langle D_H B(t) \rangle_t\;dt
\end{equation*}

\section{Construction of the free channel}\label{sec:free-channel}
The free channel wave operator $\Omega_\text{free}$ is defined by
\begin{equation}\label{eqn:free-channel-def}
    \Omega_\text{free} u_0 := \slim_{t \to \infty} e^{it\Delta} \Jfree u(t)
\end{equation}
where $\Jfree$ is the projection
\begin{equation}\label{eqn:Jfree-def}
    \Jfree = e^{-it\Delta} F_{\leq t^\alpha}(|x|) e^{it\Delta} F_{\geq t^{-\delta}}(|D|)
\end{equation}
To justify the name `free channel,' we observe that if $\Omega_\text{free} u_0$ exists, then for any $\phi \in H^1$,
\begin{equation*}\begin{split}
    \lim_{t \to \infty} \left\langle e^{-it\Delta} \phi, u(t) \right\rangle_{L^2} =& \lim_{t \to \infty} \left\langle \phi, e^{it\Delta} \Jfree u(t) \right\rangle_{L^2} + \lim_{t \to \infty} \left\langle (I - \Jfree)^*e^{-it\Delta}\phi, u(t) \right\rangle_{L^2}\\
    =& \langle \phi, \Omega_\text{free} u_0 \rangle_{L^2}
\end{split}\end{equation*}
since
\[(I - \Jfree)^*e^{-it\Delta} = F_{\geq t^{-\delta}}(|D|) e^{-it\Delta} F_{> t^\alpha}(|x|)   + e^{-it\Delta} F_{< t^{-\delta}}(|D|)\]
converges strongly to $0$ in $H^1$.  Thus, $u(t) - e^{-it\Delta} \Omega_\text{free} u_0$ is weakly orthogonal to all free waves $e^{-it\Delta} \phi$, which means that $e^{-it\Delta} \Omega_\text{free} u_0$ is the free component of $u(t)$.

The definition of $\Omega_\text{free}$ appears to depend on $\alpha$ and $\delta$ through $\Jfree$, suggesting that we should write $\Omega_\text{free}^{\alpha,\delta}$ instead of $\Omega_\text{free}$. In fact, $\Omega_\text{free}^{\alpha,\delta} = \Omega_\text{free}^{\alpha',\delta'}$ whenever both wave operators exist, so there is no ambiguity.  To see this, observe that
\begin{equation*}\begin{split}
    \Omega_\text{free}^{\alpha,\delta} u_0 - \Omega_\text{free}^{\alpha',\delta'} u_0 =& \slim_{t \to \infty} e^{it\Delta} (\Jfree^{\alpha,\delta} - \Jfree^{\alpha',\delta'}) u(t)\\
    =& \wlim_{t \to \infty} e^{it\Delta} (\Jfree^{\alpha,\delta} - \Jfree^{\alpha',\delta'}) u(t)\\
    =& \wlim_{t \to \infty} \left(F_{\leq t^\alpha}(|x|) F_{\geq t^{-\delta}}(|D|) - F_{\leq t^{\alpha'}}(|x|)  F_{\geq t^{-\delta'}}(|D|)\right) e^{it\Delta} u(t)
\end{split}\end{equation*}
Now, 
\begin{equation*}\begin{split}
    F_{\leq t^\alpha}(|x|) F_{\geq t^{-\delta}}(|D|) \!-\! F_{\leq t^{\alpha'}}(|x|)  F_{\geq t^{-\delta'}}(|D|) =\hspace{-1.5em}&\hspace{1.5em} \left(F_{\leq t^\alpha}(|x|) - F_{\leq t^{\alpha'}}(|x|) \right) F_{\geq t^{-\delta'}}(|D|)\\
    &+  F_{\geq t^{-\delta}}(|D|) - F_{\geq t^{-\delta'}}(|D|)\\
    &- F_{> t^\alpha}(|x|)  \left(F_{\geq t^{-\delta}}(|D|) - F_{\geq t^{-\delta'}}(|D|) \right)
\end{split}\end{equation*}
In particular, the adjoints of the operators on the right all converge strongly to $0$ in $L^2$, so 
\begin{equation*}
    \wlim_{t \to \infty} \left(F_{\leq t^\alpha}(|x|) F_{\geq t^{-\delta}}(|D|) - F_{\leq t^{\alpha'}}(|x|)  F_{\geq t^{-\delta'}}(|D|)\right) e^{it\Delta} u(t) = 0
\end{equation*}
and hence $\Omega_\text{free}^{\alpha,\delta} = \Omega_\text{free}^{\alpha',\delta'}$.  For this reason, we will omit the parameters $\alpha$ and $\delta$ in $\Omega_\text{free}$.  For appropriate choices of $\alpha$ and $\delta$, the free channel wave operator $\Omega_\text{free}$ exists (cf.~\cite[Theorem 18]{sofferScatteringLocalizedStates2024a}):
\begin{prop}
    For $\alpha$ and $\delta$ satisfying
    \begin{equation}\begin{split}
        0 < \alpha <& 1 \\
        0 < \delta <& \min(1/2, \alpha, 1-\alpha)
    \end{split}
    \end{equation}
    the strong limit defining the free channel wave operator in~\eqref{eqn:free-channel-def} exists in $H^1$.
\end{prop}

We argue using Cook's method.  Defining
$$\Omega_\text{free}(t) u_0 = e^{it\Delta} \Jfree u(t)$$
we see that
$$\Omega_\text{free}  = \slim_{t\to\infty} \Omega_\text{free}(t)$$
Now, 
\begin{subequations}\begin{align}
    \Omega_\text{free}(t) u_0 - \Omega_\text{free}(1) u_0 =& \int_1^t \partial_s \left(\Omega_\text{free}(s) u_0\right)\;ds\notag\\
    =& i\int_1^t  \tJfree(s) e^{is\Delta} Vu(s)\;ds\label{eqn:cooks-pot}\\
    &+ \int_1^t \partial_s \tJfree(s) e^{is\Delta} u(s)\;ds \label{eqn:channel-err}
\end{align}\end{subequations}
where $\tJfree$ is given in~\eqref{eqn:tJfree-def}
\begin{equation*}
    \tJfree(s) = F_{\leq s^\alpha}(|x|) F_{\geq s^{-\delta}}(|D|)
\end{equation*}
Following~\cite{sofferScatteringLocalizedStates2024a}, we will call~\eqref{eqn:cooks-pot} the potential term and~\eqref{eqn:channel-err} the channel restriction error term.  Since $\Omega_\text{free}(1)u_0$ is in $H^1$, it suffices to prove that the integrals giving the potential term and the channel restriction error converge strongly as $t \to \infty.$

\textbf{The potential term:} We begin by considering the potential term.  Using~\Cref{cor:loc-decay-weighted} and recalling the decay assumptions~\cref{eqn:V-decay,eqn:dV-decay}, we see that the integrand for~\eqref{eqn:cooks-pot} is $L^1$ in time, so~\eqref{eqn:cooks-pot} converges strongly as $t \to \infty$.

\textbf{The channel restriction error:} We will now show how to obtain convergence for the channel restriction error term~\eqref{eqn:channel-err}.  We begin by writing
\begin{equation*}\begin{split}
    \partial_s \tJfree(s) =& \partial_s F_{\leq s^\alpha}(|x|) F_{\geq s^{-\delta}}(|D|) + F_{\leq s^\alpha}(|x|) \partial_s F_{\geq s^{-\delta}}(|D|)\\
                            =:& A_1(s) + A_2(s)
\end{split}\end{equation*}
Focusing on $A_1(s)$ and writing $\psi(s) = e^{is\Delta} u(s)$, we see that by duality,
\begin{equation*}\begin{split}
    \left\lVert \int_1^t A_1(s) \psi(s) ds \right\rVert_{L^2} =\hspace{-2em}&\hspace{2em} \smashoperator[l]{\sup_{\lVert \phi \rVert_{L^2} = 1}}\int_1^t \langle \phi,  \partial_s F_{\leq s^\alpha}(|x|) F_{\geq s^{-\delta}}(|D|) \psi(s)\rangle ds\\
    \leq& \smashoperator[l]{\sup_{\lVert \phi \rVert_{L^2} = 1}} \left\lVert \sqrt{\partial_s F_{\leq s^\alpha}(|x|)} \phi \right\rVert_{L^2_{s,x}} \left\lVert \sqrt{\partial_s F_{\leq s^\alpha}(|x|)} F_{\geq s^{-\delta}}(|D|) \psi(s)\right\rVert_{L^2_{s,x}}
\end{split}\end{equation*}
Since $\phi$ is time-independent, the first term can be rewritten as
\begin{equation*}\begin{split}
    \left\lVert \sqrt{\partial_s F_{\leq s^\alpha}(|x|)} \phi \right\rVert_{L^2_s(1,t; L^2_x)} =& \sqrt{\left\langle \phi, \int_1^t \partial_s F_{\leq s^\alpha}(|x|) \;ds \phi \right\rangle}\\
    =& \sqrt{\langle \phi, (F_{\leq t^\alpha}(|x|) -F_{\leq 1}(|x|)) \phi \rangle}\\
    \leq& \lVert \phi \rVert_{L^2}
\end{split}\end{equation*}
Thus, it is enough to prove that $\sqrt{\partial_s F_{\leq s^\alpha}(|x|)} F_{\geq s^{-\delta}}(|D|) e^{is\Delta}u(s) \in L^2_{s,x}$.
\begin{lemma}\label{lem:pres-wave-op}
    For $u$ satisfying~\eqref{eqn:main-eqn},
    \begin{equation}\label{eqn:A1-L2-bd}
        \left\lVert \sqrt{\partial_s F_{\leq s^\alpha}(|x|)} F_{\geq s^{-\delta}}(|D|) e^{is\Delta} u(s) \right\rVert_{L^2_{s,x}} \lesssim \lVert u \rVert_{L^\infty_sL^2_x}
    \end{equation}
\end{lemma}
\begin{proof}
    To begin, we define the observable
    \begin{equation*}
        B(s) = F_{\geq s^{-\delta}}(|D|) F_{\leq s^\alpha}(|x|) F_{\geq s^{-\delta}}(|D|)
    \end{equation*}
    Noting that $\psi$ satisfies the abstract Schr\"odinger equation
    \begin{equation*}
        i\partial_s \psi = -e^{is\Delta} V e^{-is\Delta} \psi
    \end{equation*}
    and recalling~\eqref{eqn:PO-deriv} and~\eqref{eqn:H-deriv-def}, we find that
    \begin{subequations}\begin{align}
        \partial_s \langle B(s)\rangle_{s,\psi} =& \left\langle\partial_s F_{\geq s^{-\delta}}(|D|) F_{\leq s^\alpha}(|x|) F_{\geq s^{-\delta}}(|D|)\right\rangle_{s,\psi}\label{eqn:B-deriv-FD-first}\\
        &+ \left\langle F_{\geq s^{-\delta}}(|D|) F_{\leq s^\alpha}(|x|) \partial_sF_{\geq s^{-\delta}}(|D|)\right \rangle_{s,\psi}\label{eqn:B-deriv-FD-last}\\
        &+ \left\langle F_{\geq s^{-\delta}}(|D|) \partial_s F_{\leq s^\alpha}(|x|) F_{\geq s^{-\delta}}(|D|) \right \rangle_{s,\psi}\label{eqn:B-deriv-Fx}\\
        &-i \left\langle \left[e^{is\Delta} V e^{-is\Delta}, F_{\geq s^{-\delta}}(|D|) F_{\leq s^\alpha}(|x|) F_{\geq s^{-\delta}}(|D|)\right] \right \rangle_{s,\psi}\label{eqn:B-deriv-pot-comm}
    \end{align}\end{subequations}
    For~\eqref{eqn:B-deriv-pot-comm}, we note that
    \begin{equation*}\begin{split}
        \lVert V e^{-is\Delta} F_{\geq s^{-\delta}}(|D|) F_{\leq s^\alpha}(|x|) \rVert_{L^2 \to L^2} \lesssim& s^{-\frac{\sigma}{2}}
    \end{split}\end{equation*}
    by~\Cref{cor:loc-decay-weighted}.  A similar estimate holds for $F_{\leq s^\alpha}(|x|) F_{\geq s^{-\delta}}(|D|)  e^{is\Delta}  V$, so
    \begin{equation*}
        |\eqref{eqn:B-deriv-pot-comm}| \lesssim s^{-\sigma/2}\lVert u \rVert_{L^2}^2 \in L^1_s(1,\infty)
    \end{equation*}
    By using the symmetrization formula~\eqref{eqn:symm-formula}, we can rewrite the sum of~\eqref{eqn:B-deriv-FD-first} and~\eqref{eqn:B-deriv-FD-last} as a nonnegative term and a remainder:
    \begin{equation*}\begin{split}
        \eqref{eqn:B-deriv-FD-first} + \eqref{eqn:B-deriv-FD-last} =& 2 \left\langle \sqrt{F_{\geq s^{-\delta}} \partial_s F_{\geq s^{-\delta}}(|D|)} F_{\leq s^\alpha}(|x|) \sqrt{F_{\geq s^{-\delta}} \partial_s F_{\geq s^{-\delta}}(|D|)} \right\rangle_{s,\psi}\\
        &+ \langle R(s) \rangle_{s,\psi}
    \end{split}\end{equation*}
    Here, the remainder $R(s)$ involves double commutators of $F_{\leq s^\alpha}(|x|)$ with $F_{\geq s^{-\delta}}(|D|
    )$ and $\partial_s F_{\geq s^{-\delta}}(|D|)$.  Each commutator gains a power of  $s^{\delta - \alpha}$, so since 
    $$\lVert \partial_s F_{\geq s^{-\delta}}(|D|) \rVert_{L^2 \to L^2} \lesssim s^{-1}$$
    we have that
    \begin{equation*}
        |\langle R(s) \rangle_{s,\psi}| \lesssim s^{-1+2(\delta - \alpha)}\lVert u(s) \rVert_{L^2}^2 \in L^1_s(1,\infty)
    \end{equation*}
    since $\delta < \alpha$.
    
    Finally, we note that
    \begin{equation*}
        \eqref{eqn:B-deriv-Fx} = \left\lVert \sqrt{\partial_s F_{\leq s^\alpha}(|x|)} F_{\geq s^{-\delta}}(|D|) e^{is\Delta} u(s) \right\rVert_{L^2}^2
    \end{equation*}
    is manifestly nonnegative.  Collecting these results, we conclude that
    \begin{equation*}
        \partial_s \langle B(s) \rangle_{s,\psi} + O_{L^1_s(1,\infty)}(\lVert u \rVert_{L^\infty_sL^2_x}^2) \geq \left\lVert \sqrt{\partial_s F_{\leq s^\alpha}(|x|)} F_{\geq s^{-\delta}}(|D|) e^{is\Delta} u(s) \right\rVert_{L^2_x}^2
    \end{equation*}
    Integrating and using the fact that $|\langle B(s) \rangle_{s,\psi}|
    \leq \lVert u(s) \rVert_{L^2}^2$ now gives~\eqref{eqn:A1-L2-bd}.
\end{proof}
Based on the earlier duality argument, we conclude that
\begin{equation*}
    \int_1^t A_1(s) e^{is\Delta} u(s)\;ds 
\end{equation*}
converges strongly in $L^2$.  In order to upgrade the convergence to $H^1$, we note that
\begin{equation*}
    \lVert [\partial_x, A_1(s)] \rVert_{L^2 \to L^2} \lesssim s^{-1-\alpha} \in L^1_s(1,\infty)
\end{equation*}
so it suffices to show that $\int_1^t A_1(s) e^{is\Delta} \partial_x u(s)\;ds$ converges strongly in $L^2$.  This follows from essentially the same argument as before once we use the propagation observable $-\partial_x B(s) \partial_x$ in place of $B(s)$.  Moreover, the terms involving $A_2(s)$ can be bounded by essentially the same argument as above once we note that
\begin{equation*}
    A_2(s) = \partial_s F_{\geq s^{-\delta}}(|D|) F_{\leq s^{\alpha}}(|x|) + O(s^{-1+\delta -\alpha})
\end{equation*}

\section{Energy localization}\label{sec:energy-loc}

The arguments in~\Cref{sec:free-channel} let us write
\begin{equation*}
    u(t) = u_{\textup{wb}}(t) + e^{-it\Delta} u_+ + o_{H^1}(1)
\end{equation*}
for some $u_+ \in H^1$, where the weakly bound component $u_{\textup{wb}}(t) = (I-\Jfree) u(t)$ has the property that for any $\phi \in H^1$,
\begin{equation*}
    \lim_{t \to \infty} \langle e^{-it\Delta}\phi, u_{\textup{wb}}(t) \rangle = 0
\end{equation*}
In order to prove the first part of~\Cref{thm:main-thm}, we will now show that  we have the further decomposition
\begin{equation}\label{eqn:uwb-decomp}
    u_\textup{wb}(t) = u_\textup{loc}(t) + o_{\dot{H}^1}(1)
\end{equation}
as $t \to \infty$.  To do this, we will first define incoming/outgoing projections, which allow us to obtain stronger dispersive decay estimates for localized functions in the favorable time direction.  Then, using these projections, we carefully decompose $u_{\text{wb}}(t)$ in a localized part and a part that decays in $\dot{H}^1$.

The basic plan for the rest of the article is as follows: In~\Cref{sec:in-out}, we will define the incoming and outgoing projectors $\Pin_j$ and $\Pout_j$ and their essential properties.  Next, we will use these projectors in~\Cref{sec:uwb-decom} to decompose $\uwb(t)$ into a remainder component $\urem$ and the localized component $\uloc$, and show in~\Cref{sec:urem-bd} that the $\urem(t) \to 0$ in $\dot{H}^1$ as $t \to \infty$ and, in~\Cref{sec:u-loc-bd-subsec} that $\uloc$ satisfies~\eqref{eqn:uloc-desired}.  Finally, in~\Cref{sec:higher-order-loc-bds} we will prove that $\uloc$ can be re-defined to satisfying the higher order energy localization condition~\eqref{eqn:uloc-higher-order}.

\subsection{Incoming/outgoing projection}\label{sec:in-out}

For $\theta \in (0,1)$ and $j > 0$ we define the projection operators
\begin{equation}\label{eqn:in-out-proj}\begin{split}
    \Pin_{j} =& F_{2^j}(x) F_{\geq 2^{-\theta j}}\left(D\right) + F_{2^j}(-x) F_{\geq 2^{-\theta j}}\left(-D\right)\\
    \Pout_{j} =& F_{2^j}(x) F_{\geq 2^{-\theta j}}\left(-D\right) + F_{2^j}(-x) F_{\geq 2^{-\theta j}}\left(D\right)\\
    \Plow_{j} =& F_{2^j}(|x|)F_{< 2^{-\theta j}}\left(|D|\right)\\
    \Pin_0 =& F_{\leq 1}(|x|) F_{\geq 1}(|D|)\\
    \Pout_0 =& 0\\
    \Plow_0 =& F_{\leq 1}(|x|) F_{< 1}(|D|)\\
    \Pin_{\geq j} =& F_{\geq 2^j}(x) F_{\geq 2^{-\theta j}}\left(D\right) + F_{\geq 2^j}(-x) F_{\geq 2^{-\theta j}}\left(-D\right)\\
    \Pout_{\geq j} =& F_{\geq 2^j}(x) F_{\geq 2^{-\theta j}}\left(-D\right) + F_{\geq 2^j}(-x) F_{\geq 2^{-\theta j}}\left(D\right)\\
    \Plow_{\geq j} =& F_{\geq 2^j}(|x|)F_{< 2^{-\theta j}}\left(|D|\right)
\end{split}\end{equation}

The incoming/outgoing projectors allow us to gain one derivative and $\theta$ space weights when applied in the favorable time direction to a localized function:
\begin{lemma}\label{lem:in-out-lem}
    Let $n$ be a positive integer.  For $f$ such that $\jBra{x}^\sigma f \in L^\infty_t H^{S'}_x$ with $\sigma > 2$ and $S' > S \geq 0$,
    \begin{equation*}\begin{split}
        \sum_{j= 0}^\infty \left\lVert \jBra{x}^{n\theta}\partial_x^n \Pin_{j} \int_0^t e^{-i(t-s)\Delta} f(s)\;ds \right\rVert_{H^{S}} \lesssim_{S, S'}&  \sum_{m=0}^{n-1}\lVert \jBra{x}^{\sigma+m\theta}\partial_x^m  f \rVert_{L^\infty_t H^{S'}_x}\\
        \sum_{j= 0}^\infty \left\lVert \jBra{x}^{n\theta} \partial_x^n \Pout_{j} \int_t^\infty e^{-i(t-s)\Delta} f(s)\;ds \right\rVert_{H^{S}} \lesssim_{S, S'}& \sum_{m=0}^{n-1} \lVert \jBra{x}^{\sigma+m\theta} \partial_x^m f \rVert_{L^\infty_t H^{S'}_x}
    \end{split}\end{equation*}
\end{lemma}

\begin{rmk}
    If $f(s) = f$ is a time-independent function whose $L^2$ mass is localized in the ball $B_1(0)$ (up to rapidly decaying tails) physical space and whose a Fourier transform that is concentrated in the region $|\xi| \sim 2^k$, then heuristically $e^{-i\tau \Delta} f$ will concentrate in physical space in the annulus $|x| \sim 2^k|\tau|$ for $2^k\tau \gg 1$.  In particular, for $j \gg 1$, we expect that schematically
    \begin{equation*}
        \Pin_0 \int_0^t e^{-i(t-s)\Delta} f\;ds \approx \Pin_0 \int_{t-C2^{-k}}^t e^{-i(t-s)\Delta} f\;ds
    \end{equation*}
    since the remaining contribution from $0 \leq s \leq t - C2^{-k}$ will be lower order by the nonstationary phase principle.  The only smallness for the leading term comes from the short length of the time integration, so it is only possible to gain one derivative.  (See the bounds for $A_k$ and $B_k$ below for a more quantitative version of this argument.).  
\end{rmk}

\begin{proof}
    We will focus on the case $S = 0$, since $S > 0$ can be handled similarly by commuting the extra derivative through the projection.  We will also focus on the incoming projections, and for notational simplicity we will first give the argument for $j > 0$, then briefly discuss the modifications necessary to address the case $j = 0$.  
    
    To begin, we project to frequencies of size $2^k \gtrsim 2^{-\theta j}$ and decompose the integral as:
    \begin{subequations}\begin{align}
        \jBra{x}^{n\theta}\partial_x^n \Pin_j P_k \!\!\int_0^t \!\!e^{-i(t-s)\Delta} f(s) ds =&\jBra{x}^{n\theta}\partial_x^n \Pin_j P_k \!\! \int_0^t \!\!e^{-i(t-s)\Delta} F_{\leq 2^{j-10}}(|x|)  f ds\label{eqn:x-dx-Pin-bd-1}\\
        &\hspace{-4em}+ \jBra{x}^{n\theta}\partial_x^n \Pin_j P_k \int_{{t-C2^{j-k}}}^t \!\!\!\!\!\!e^{-i(t-s)\Delta} F_{> 2^{j-10}}(|x|)   f ds\label{eqn:x-dx-Pin-bd-2}\\
        &\hspace{-4em}+\jBra{x}^{n\theta}\partial_x^n \Pin_j P_k \int_0^{{t-C2^{j-k}}} \!\!e^{-i(t-s)\Delta} F_{> 2^{j-10}}(|x|)   f ds\label{eqn:x-dx-Pin-bd-3}
    \end{align}\end{subequations}
    for a large but fixed constant $C$.  Schematically, the terms~\cref{eqn:x-dx-Pin-bd-1,eqn:x-dx-Pin-bd-2,eqn:x-dx-Pin-bd-3} can be understood as follows:
    \begin{itemize}
        \item The term~\eqref{eqn:x-dx-Pin-bd-1} represents the contribution from waves that are created near the origin by the forcing $f$, evolve forward in time under the free Schr\"odinger evolution, and are then observed to be traveling \textit{inward} at some large distance from the origin.  Evidently, this behavior would be forbidden for classical particle. For the quantum system we consider, we can use stationary phase to prove that these `classically forbidden' interactions decay arbitrarily fast in time and gain an arbitrary number of derivative and spatial weights, which is more than sufficient.  It is only for this term that the precise (microlocal) relationship between $x$ and $\xi$ on the support of $\Pin_j$ enters; the other terms are estimated in a way this insensitive to the direction of time and the sign of the frequency variable.
        \item The term~\eqref{eqn:x-dx-Pin-bd-2} is a `recent past' Duhamel term accounting for the contribution of the forcing $f$ far from the origin.  Due to the shortness of the time interval and the fact that the forcing occurs far from the origin, our only choice is to use the shortness of the time integral and the decay and smoothness of $f$ to compensate for the space weights and derivatives.
        \item The term~\eqref{eqn:x-dx-Pin-bd-3} is a `distant past' Duhamel term which accounts for forcing far from the origin.  Because of the largeness of $t-s$, dispersive effects are strong enough to be of use, so we further subdivide the forcing into a region at a distance $\ll (t-s)2^k$ from the origin and the complementary region where $|x| \gtrsim (t-s)2^k$.  In the second region, we can exploit the decay of $f$ directly to control the integral.  For the first region, we take advantage of the fact that waves at frequency $\sim 2^k$ travel with speed $\sim 2^k$, so the majority of waves produced in this region will be outside the spatial support of $\Pin_j$.
    \end{itemize}
    We now show how to bound each of these terms rigorous and such that the bound is summable in $j$ and $K$.  For~\eqref{eqn:x-dx-Pin-bd-1}, we note that
    \begin{equation}\label{eqn:x-dx-Pin-Pk-ident}
        \jBra{x}^{n\theta}\partial_x^n \Pin_j P_k = \jBra{x}^{n\theta}\Pin_j \partial_x^n P_k + \jBra{x}^{n\theta}[\partial_x^n, \Pin_j] P_k
    \end{equation}
    The commutator term will be lower order, since $[\partial_x, \Pin_j]$ has the same phase-space support properties as $\Pin_j$ but has size $2^{-j} \lesssim 2^k$.  Thus, we will focus our attention on the terms containing $\jBra{x}^{n\theta}\Pin_j \partial_x^n P_k$.  For~\eqref{eqn:x-dx-Pin-bd-1}, we re-express this leading order term as an integration against a kernel:
    \begin{equation*}
        \eqref{eqn:x-dx-Pin-bd-1} = \int_0^t \int K_{j,k}(x,y;t-s)  f(y,s)\;dyds + \{\textup{better terms}\}
    \end{equation*}
    where $\{\textup{better terms}\}$ denotes terms involving commutators of $\partial_x$ and $\Pin_j$, which have better decay properties, and 
    \begin{equation*}
        K_{j,k}(x,y;\tau) = \frac{\jBra{x}^{n\theta} F_{2^j}(x) F_{\leq  2^{j-10}}(|y|)}{2\pi} \int F_{2^k}(\xi)(i\xi)^n e^{i(\tau\xi^2+\xi(x-y))}\;d\xi + \{\textup{similar}\}
    \end{equation*}
    where $\{\textup{similar}\}$ denotes a term microlocalized to $-x \sim 2^j$, $-\xi \sim 2^k$.  Based on the microlocalization of $x,y$ and $\xi$, the derivative of the phase satisfies
    \begin{equation*}
        |\partial_\xi (\tau\xi^2+\xi(x-y))| = |2\tau \xi + (x-y)| \gtrsim \tau 2^k + 2^j > 0
    \end{equation*}
    so we can integrate by parts repeatedly to obtain
    \begin{equation*}\begin{split}
        |K_{j,k}| \lesssim_N& \frac{\jBra{x}^{n\theta} F_{2^j}(x) F_{\leq  2^{j-10}}(|y|)}{2\pi} \sum_{a+b=N} \int\frac{\tau^a}{|2\tau\xi + (x-y)|^{N+a}} 2^{(n-b)k} |F^{(b)}_{2^k}(|\xi|)|\;d\xi\\
        &+ \{\textup{similar}\}\\
        \lesssim_N& \frac{2^{n\theta j} F_{2^j}(x) F_{\leq  2^{j-10}}(|y|)}{2\pi} \sum_{a+b=N} \int\frac{\tau^a}{|2\tau\xi + (x-y)|^{N+a}} 2^{(n-b)k} |F^{(b)}_{2^k}(|\xi|)| \;d\xi\\
        &+ \{\textup{similar}\}\\
        \lesssim_N& F_{2^j}(x) F_{\leq  2^{j-10}}(|y|) 2^{n\theta j+ (n+1-N)k} \min\left(\frac{1}{2^{Nk}\tau^N}, 2^{-Nj}\right)
    \end{split}\end{equation*}
    Thus,
    \begin{equation*}
        \lVert\eqref{eqn:x-dx-Pin-bd-1}\rVert_{L^2}  \leq \lVert K_{j,k}(x,y;\tau)\rVert_{L^1_{\tau}(0,\infty; L^2_{x,y})} \lVert f \rVert_{L^\infty_t L^2_x}
    \end{equation*}
    Based on the previous calculation, we see that for $N \geq 2$
    \begin{equation*}
        \lVert K_{j,k}(x,y;\tau) \rVert_{L^2_{x,y}} \lesssim_N 2^{(n\theta+1)j + (n+1-N)k} \begin{cases}
            2^{-Nk} \tau^{-N}, & \tau > 2^{j-k}\\
            2^{-Nj},& 0 \leq \tau \leq 2^{j-k}
        \end{cases}
    \end{equation*}
    so
    \begin{equation*}\begin{split}
        \lVert\eqref{eqn:x-dx-Pin-bd-1}\rVert_{L^2} \lesssim& 2^{(n\theta+2-N)j} 2^{(n-N)k} \lVert f \rVert_{L^\infty_t L^2_x}
    \end{split}\end{equation*}
    Assuming that $N > n$ and summing over $k \geq -\theta j - 10$, we find that
    \begin{equation*}
        \sum_{k \geq -\theta j - 10} \lVert \eqref{eqn:x-dx-Pin-bd-1}\rVert_{L^2} \lesssim 2^{(2 - (1-\theta)N)j} \lVert f \rVert_{L^\infty_tL^2_x}
    \end{equation*}
    Taking $N > \max\left(n, \frac{2}{1-\theta}\right)$ and summing in $j$ gives the required bound for this term.  Turning to~\eqref{eqn:x-dx-Pin-bd-2}, we again commute the derivatives through the projection $\Pin_j$ to write
    \begin{subequations}\begin{align}
         \eqref{eqn:x-dx-Pin-bd-2} =& \jBra{x}^{n\theta} \Pin_j \partial_x^n P_k \int_{t-C2^{j-k}}^t e^{-i(t-s)\Delta} F_{> 2^{j-10} }(|x|) f\;ds + \{\textup{better}\}\notag\\
         =& \jBra{x}^{n\theta} \Pin_j \partial_x P_k \int_{t-C2^{j-k}}^t e^{-i(t-s)\Delta} F_{> 2^{j-10}}(|x|) \partial_x^{n-1} f\;ds\label{eqn:x-dx-Pin-bd-2-1}\\
         &+ \jBra{x}^{n\theta} \Pin_j \partial_x P_k \int_{t-C2^{j-k}}^t e^{-i(t-s)\Delta} [\partial_x^{n-1}, F_{> 2^{j-10}}(|x|)] f\;ds\label{eqn:x-dx-Pin-bd-2-2}\\
        &+ \{\textup{better}\}\notag
    \end{align}\end{subequations}
    We now claim that
    \begin{equation*}
        \lVert \eqref{eqn:x-dx-Pin-bd-2-1} \rVert_{L^2} \lesssim 2^{(1 + \theta - \sigma)j} \lVert P_k \jBra{x}^{\sigma + (n-1)\theta} \partial_x^{n-1} f \rVert_{L^2}
    \end{equation*}
    and
    \begin{equation*}
        \lVert \eqref{eqn:x-dx-Pin-bd-2-2} \rVert_{L^2} \lesssim 2^{(1 + \theta -\sigma)j} \sum_{m=0}^{n-2}\lVert P_k \jBra{x}^{\sigma + m\theta} \partial_x^{m} f \rVert_{L^2}
    \end{equation*}
    For~\eqref{eqn:x-dx-Pin-bd-2-1}, we begin by dyadically decomposing the integrand in space:
    \begin{equation*}
        \eqref{eqn:x-dx-Pin-bd-2-1} = \sum_{\ell > j - 10} \jBra{x}^{n\theta} \Pin_j \partial_x P_k \int_{t-C2^{j-k}}^t e^{-i(t-s)\Delta} F_{2^\ell} (|x|) \partial_x^{n-1} f\;ds
    \end{equation*}
    Using the fact that the Schr\"odinger evolution is unitary, we find that
    \begin{equation*}
        \lVert \eqref{eqn:x-dx-Pin-bd-2-1} \rVert_{L^2} \leq \sum_{\ell > j-10} 2^{(n\theta+1)j} \lVert P_k F_{2^\ell} (|x|) \partial_x^{n-1} f \rVert_{L^\infty_t L^2_x}
    \end{equation*}
    Now, writing
    \begin{equation*}
        \lVert P_k F_{2^\ell} \partial_x^{n-1} f \rVert_{L^\infty_t L^2_x} \leq \lVert F_{\sim 2^\ell} P_k F_{2^\ell}\partial_x^{n-1} f \rVert_{L^\infty_t L^2_x} + \lVert (1-F_{\sim 2^\ell}) P_k F_{2^\ell} \partial_x^{n-1} f \rVert_{L^\infty_t L^2_x}
    \end{equation*}
    we see from~\Cref{cor:full-freq-proj-supp-lem} that
    \begin{equation*}
        \lVert (1-F_{\sim 2^\ell}) P_k F_{2^\ell} \partial_x^{n-1} f \rVert_{L^\infty_t L^2_x} \lesssim_N 2^{-N(\ell+k)} 2^{-((n-1)\theta + \sigma)\ell} \lVert \jBra{x}^{\sigma+(n-1)\theta} \partial_x^{n-1} f \rVert_{L^\infty_t L^2_x}
    \end{equation*}
    which can be summed in $\ell$ and $k$ for $N = 1$.  For the other term, we have that
    \begin{equation*}\begin{split}
        \lVert F_{\sim 2^\ell} P_k F_{2^\ell}\partial_x^{n-1} f \rVert_{L^\infty_t L^2_x} \lesssim& 2^{-(\sigma + (n-1)\theta)\ell} \lVert F_{\lesssim 2^\ell} \jBra{x}^{\sigma + (n-1)\theta} P_k F_{2^\ell}\partial_x^{n-1} f \rVert_{L^\infty_t L^2_x}\\
        \lesssim& 2^{-(\sigma + (n-1)\theta)\ell} \lVert P_k \jBra{x}^{\sigma + (n-1)\theta} F_{2^\ell}\partial_x^{n-1} f \rVert_{L^\infty_t L^2_x}\\
        &+ 2^{-(\sigma + (n-1)\theta)\ell} \lVert F_{\lesssim 2^\ell} [\jBra{x}^{\sigma + (n-1)\theta}, P_k] F_{2^\ell}\partial_x^{n-1} f \rVert_{L^\infty_t L^2_x}
    \end{split}\end{equation*}
    Using~\Cref{cor:comm-cor} with $\mu = 1 < a = \sigma + (n-1)\theta$, we see that the commutator term is bounded by
    \begin{equation*}\begin{split}
        2^{-(\sigma + (n-1)\theta)\ell} \lVert F_{\lesssim 2^\ell} [\jBra{x}^{\sigma + (n-1)\theta}, P_k] F_{2^\ell}\partial_x^{n-1} f \rVert_{L^\infty_t L^2_x} \lesssim\hspace{-1in}&\hspace{1in} 2^{-\ell} 2^{-k} \lVert F_{2^\ell}\partial_x^{n-1} f \rVert_{L^\infty_t L^2_x}\\
        \lesssim& 2^{-(1 + \sigma + (n-1)\theta)\ell} 2^{-k} \lVert \jBra{x}^{\sigma + (n-1)\theta} \partial_x^{n-1} f \rVert_{L^\infty_t L^2_x}
    \end{split}\end{equation*}
    Thus, summing in $\ell$ gives that
    \begin{equation*}\begin{split}
        \lVert \eqref{eqn:x-dx-Pin-bd-2-1} \rVert_{L^2} \lesssim& 2^{-(\sigma - 1 - \theta)j} \lVert P_k \jBra{x}^{\sigma + (n-1)\theta} \partial_x^{n-1} f \rVert_{L^\infty_t L^2_x}\\
        &+ 2^{-k} 2^{-(\sigma -\theta)j} \lVert \jBra{x}^{\sigma + (n-1)\theta} \partial_x^{n-1} f \rVert_{L^\infty_t L^2_x}\\
        \lesssim& 2^{-(\sigma - 1 - \theta)j} \min(1, 2^{-S'k}) \lVert \jBra{x}^{\sigma + (n-1)\theta} \partial_x^{n-1} f \rVert_{L^\infty_t H^{S'}_x}\\
        &+ 2^{-k} 2^{-(\sigma -\theta)j} \lVert \jBra{x}^{\sigma + (n-1)\theta} \partial_x^{n-1} f \rVert_{L^\infty_t L^2_x}
    \end{split}\end{equation*}
    Summing this over $j > 0$, $k \geq -\theta j - 10$ gives
    \begin{equation*}
        \sum_{j = 1}^\infty \sum_{k \geq -\theta j - 10} \lVert \eqref{eqn:x-dx-Pin-bd-2-1} \rVert_{L^2} \lesssim \lVert \jBra{x}^{\sigma+(n-1)\theta} \partial_x^{n-1} f \rVert_{L^\infty_t H^{S'}_x}
    \end{equation*}
    as required.  To handle the term~\eqref{eqn:x-dx-Pin-bd-2-2}, we simply observe that the commutator term can be rewritten as
    $$[\partial_x^{n-1}, F_{> 2^{j-10}}(|x|)] = \sum_{m=0}^{n-2}\left(\frac{\sgn(x)}{2^{j-9}}\right)^{n-1-m} {n-1 \choose m} F_{> 2^{j-10}}^{(n-1-m)}(|x|) \partial_x^{m}$$
    and each individual term can be estimated as above.
    
    Finally, for~\eqref{eqn:x-dx-Pin-bd-3}, we decompose the integral as
    \begin{subequations}\begin{align}
        \eqref{eqn:x-dx-Pin-bd-3} =& \jBra{x}^{n\theta} \Pin_j \partial_x P_k \int_0^{t-C2^{j-k}} e^{-i(t-s)\Delta} F_{> (t-s) 2^{k-30}}(|x|) \partial_x^{n-1} f\;ds\label{eqn:x-dx-Pin-bd-3-1}\\
        &+ \jBra{x}^{n\theta} \Pin_j \partial_x P_k \int_0^{{t-C2^{j-k}}} e^{-i(t-s)\Delta} F_{\leq (t-s) 2^{k-30}}(|x|) F_{> 2^{j-10}}(|x|) \partial_x^{n-1} f\;ds\label{eqn:x-dx-Pin-bd-3-2}\\
        &+ \{\textup{better}\}\notag
    \end{align}\end{subequations}
    Repeating the dyadic decomposition argument we used to control~\eqref{eqn:x-dx-Pin-bd-2-1}, we find that
    \begin{equation*}\begin{split}
        \lVert \eqref{eqn:x-dx-Pin-bd-3-1} \rVert_{L^2} \lesssim& 2^{n\theta j} 2^k \int_0^{{t-C2^{j-k}}} \frac{\lVert P_k\jBra{x}^{\sigma + (n-1)\theta} \partial_x^{n-1} f \rVert_{L^\infty_t L^2_x}}{(t-s)^{\sigma +(n-1)\theta} 2^{(\sigma + (n-1)\theta)k}}\;ds\\
        &\hspace{-2em}+ 2^{(1/2 +(n-1/2)\theta) j} 2^{(1/2+\theta/2)k} \int_0^{{t-C2^{j-k}}} \frac{\lVert \jBra{x}^{\sigma + (n-1)\theta} \partial_x^{n-1} f \rVert_{L^\infty_t L^2_x}}{(t-s)^{(\sigma +(n-1)\theta)} 2^{(\sigma + (n-1)\theta)k}}\;ds\\
        \lesssim&  2^{(1+\theta - \sigma)j}\lVert P_k\jBra{x}^{\sigma + (n-1)\theta} \partial_x^{n-1} f \rVert_{L^\infty_t L^2_x}\\
        &+ 2^{(3/2+\theta/2 - \sigma)j} 2^{-(1-\theta)/2 k} \lVert \jBra{x}^{\sigma + (n-1)\theta} \partial_x^{n-1} f \rVert_{L^\infty_t L^2_x}\\
    \end{split}\end{equation*}
    which can be summed in $j$ and $k$ to get the required bound.  Turning to~\eqref{eqn:x-dx-Pin-bd-3-2}, we observe that by~\Cref{lem:proj-fast-decay-2} and duality, we have the bound
    \begin{equation*}\begin{split}
        \lVert \eqref{eqn:x-dx-Pin-bd-3-2} \rVert_{L^2} \lesssim& 2^{n\theta j} \int_0^{t-C2^{j-k}} 2^{j/2}\frac{2^{(5/2 - 2N)k}}{(t-s)^{N-1/2}} \lVert F_{> 2^{j-10}}(|x|) \partial_x^{n-1} f \rVert_{L^\infty_t L^2_x}\;ds\\
        \lesssim_N& 2^{(2+\theta - \sigma - N)j} 2^{(1-N)k} \lVert\jBra{x}^{\sigma + (n-1)\theta} \partial_x^{n-1} f \rVert_{L^\infty_t L^2_x}
    \end{split}\end{equation*}
    which also gives the required bounds after summing, provided we take $N$ large enough.

    It only remains to bound the term with the $\Pin_0$ projection (recall that $\Pout_0 = 0$ by definition).  The main difference here is that $\Pin_0$ projects to nonsmall frequencies, but does not retain any information about the (microlocal) sign of $x \cdot \xi$.  To deal with this term, we first observe that by the spatial localization of $\Pin_0$,
    \begin{equation*}\begin{split}
        \left\lVert \jBra{x}^{n\theta} \partial_x^n \Pin_0 \int_0^t \!\!\!e^{-i(t-s)\Delta} f(s) ds \right\rVert_{L^2_x} \lesssim& \left\lVert  F_{\leq 2^{10}}(|x|) \partial_x \int_0^t \!\!P_{\geq 0} e^{-i(t-s)\Delta} \partial_x^{n-1}f(s) ds \right\rVert_{L^2_x}\\
        &\hspace{-3em}+ \sum_{m=0}^{n-1} \left\lVert  F_{\leq 2^{10}}(|x|)  \int_0^t\!\! P_{\geq 0} e^{-i(t-s)\Delta} \partial_x^m f(s) ds \right\rVert_{L^2_x}
    \end{split}\end{equation*}
    We will show how to deal with the leading term, since the other terms follow from similar arguments.  For this term, we decompose as
    \begin{equation*}\begin{split}
        F_{\leq 2^{10}}(|x|) \partial_x \int_0^t P_{\geq 0} e^{-i(t-s)\Delta} \partial_x^{n-1}f(s)\;ds =& \sum_{k=0}^\infty (A_k + B_k) \\
        &+ \sum_{\ell=21}^\infty \sum_{k=0}^\infty (C_{\ell,k} + D_{\ell,k})
    \end{split}\end{equation*}
    where
    \begin{equation*}\begin{split}
        A_k =& F_{\leq 2^{10}}(|x|) \partial_x \int_0^{t-C2^{-k}} P_{k} e^{-i(t-s)\Delta} F_{\leq 2^{20}}(|x|) \partial_x^{n-1}f(s)\;ds\\
        B_{k} =&  F_{\leq 2^{10}}(|x|) \partial_x \int_{t-C2^{-k}}^t P_{k} e^{-i(t-s)\Delta} F_{\leq 2^{20}}(|x|) \partial_x^{n-1}f(s)\;ds\\
        C_{\ell,k} =& F_{\leq 2^{10}}(|x|) \partial_x \int_0^{t-C2^{\ell-k}} P_{k} e^{-i(t-s)\Delta} F_{2^\ell}(|x|) \partial_x^{n-1}f(s)\;ds\\
        D_{\ell,k} =&  F_{\leq 2^{10}}(|x|) \partial_x \int_{t-C2^{\ell-k}}^t P_{k} e^{-i(t-s)\Delta} F_{2^\ell}(|x|) \partial_x^{n-1}f(s)\;ds
    \end{split}\end{equation*}
    for some large but unspecified constant $C$.  By simply using the crude volume bound for the time integral and the estimate $\lVert P_k \partial_x \rVert_{L^2 \to L^2} \lesssim 2^k$, we see that
    \begin{equation*}\begin{split}
        \lVert B_k \rVert_{L^2} \lesssim& \lVert P_{\sim k} F_{\leq 2^{20}}(|x|) \partial_x^{n-1}f(s) \rVert_{L^\infty_s L^2_x}\\
        \lesssim& 2^{-S'k} \lVert \partial_x^{n-1}f(s) \rVert_{L^\infty_s H^{S'}_x}
    \end{split}\end{equation*}
    which is summable in $k$.  Turning to $A_k$, we observe that the kernel $K_{t-s}(x,y)$ of
    \begin{equation*}
        F_{\leq 2^{10}}(|x|) P_k \partial_x e^{-i(t-s)\Delta} F_{\leq 2^{20}}(|x|)
    \end{equation*}
    is given by the oscillatory integral
    \begin{equation*}
        K_{t-s}(x,y) = \frac{iF_{\leq 2^{10}}(|x|) F_{\leq 2^{20}}(|y|)}{2\pi} \int e^{i\left((t-s)\xi^2 + (x-y)\xi\right)}F_{2^k}(|\xi|) \xi \;d\xi
    \end{equation*}
    For $t-s \gg 2^{-k}$, the phase is nonstationary, and repeated integration by parts using the identity
    \begin{equation*}
        \mathcal{L} e^{i\left((t-s)\xi^2 + (x-y)\xi\right)} = e^{i\left((t-s)\xi^2 + (x-y)\xi\right)}
    \end{equation*}
    with $\mathcal{L} = \frac{-i}{2(t-s)\xi + (x-y)}$ gives
    \begin{equation*}
        K_{t-s}(x,y) = \frac{iF_{\leq 2^{10}}(|x|) F_{\leq 2^{20}}(|y|)}{2\pi} \int e^{i\left((t-s)\xi^2 + (x-y)\xi\right)} \left(\mathcal{L}^*\right)^N \left(F_{2^k}(|\xi|) \xi\right)\;d\xi
    \end{equation*}
    which is bounded pointwise by $(t-s)^{-N} 2^{(2-2N)k}$.  Since $K_{t-s}(x,y)$ is supported in the region $|x|, |y| \lesssim 1$, we see that taking $N = 4$ gives
    \begin{equation*}
        \lVert A_k \rVert_{L^2_x} \lesssim 2^{-k} \lVert \partial_x^{n-1} f(s) \rVert_{L^\infty_s L^2_x}
    \end{equation*}
    which is sufficient.  A small modification of this argument shows that
    \begin{equation*}
        \lVert C_{\ell,k} \rVert_{L^2} \lesssim 2^{-5\ell/2} 2^{-k} \lVert \partial_x^{n-1} f(s) \rVert_{L^\infty_s L^2_x}
    \end{equation*}
    which is jointly summable in $\ell$ and $k$.  Finally, for $D_{\ell,k}$, we use the crude volume bound
    \begin{equation*}
        \lVert D_{\ell,k} \rVert_{L^2} \lesssim 2^{\ell} \lVert P_{\sim k} F_{2^{\ell}}(|x|) \partial_x^{n-1} f(s) \rVert_{L^\infty_s L^2_x}
    \end{equation*}
    and by modifying the argument used to control~\eqref{eqn:x-dx-Pin-bd-2-1}, we obtain the bound
    \begin{equation*}
        \lVert D_{\ell,k} \rVert_{L^2} \lesssim 2^{(1-\sigma)\ell} 2^{-S' k} \lVert \jBra{x}^{\sigma}  \partial_x^{n-1} f(s) \rVert_{L^\infty_s H^{S'}_x}
    \end{equation*}
    
\end{proof}

The projectors $\Pin_j$ and $\Pout_j$ also have good summability properties.  To make this notion precise, let us define the cut-off summed projectors:
\begin{equation}\label{eqn:summed-in-out-proj}
    \cPin_J := \Pin_{\geq J} + \sum_{j=0}^{J-1} \Pin_j,\qquad\qquad \cPout_J := \Pout_{\geq J} + \sum_{j=0}^{J-1} \Pout_j
\end{equation}
Then, the summed projectors are bounded operators on $H^1$ with operator norms uniformly bounded in $J$:
\begin{lemma}\label{lem:in-out-sum-bdds}
The bounds
\begin{equation}\label{eqn:Pin-sum-CS}
    \left\lVert \cPin_J \right\rVert_{H^1 \to H^1} \leq C
\end{equation}
and
\begin{equation}\label{eqn:Pout-sum-CS}
    \left\lVert \cPout_J \right\rVert_{H^1 \to H^1} \leq C
\end{equation}
hold with a constant $C$ independent of $J \geq 0$. 
\end{lemma}
\begin{proof}
We will prove the bound~\eqref{eqn:Pin-sum-CS}: the bound for the summed outgoing projectors follows similarly.  The functions $F_{2^j}(\pm x)$ and $F_{2^{j'}}(\pm x)$ have disjoint supports for $|j - j'| \gg 1$, so we have that
\begin{equation*}
    \left\lVert \cPin_J \phi\right\rVert_{H^1}^2 \lesssim \left\lVert \Pin_{\geq J} \phi\right\rVert_{H^1}^2 + \sum_{j = 0}^{J-1} \left\lVert \Pin_j \phi\right\rVert_{H^1}^2
\end{equation*}
Since $\lVert \Pin_{\geq J} \rVert_{H^1 \to H^1}$ is bounded independent of $J$, the result will follow once we show that
\begin{equation*}
    \left\lVert \Pin_j \phi\right\rVert_{\ell^2_j H^1_x} \lesssim \lVert \phi \rVert_{H^1}
\end{equation*}
For fixed $j$, we have that
\begin{equation*}\begin{split}
    \Pin_j \phi =& \Pin_j F_{\sim 2^j}(|x|) \phi +  \Pin_j F_{\gg 2^j}(|x|) \phi + \Pin_j F_{\ll 2^j}(|x|) \phi\\
    =:& \Pin_{j,\sim} \phi + \Pin_{j, \gg} \phi + \Pin_{j, \ll} \phi
\end{split}\end{equation*}
For the $\Pin_{j,\sim} \phi $, we see at once that
\begin{equation*}
    \lVert \Pin_{j,\sim} \phi \rVert_{\ell^2_j H^1} \lesssim \lVert F_{\sim 2^j}(|x|) \phi \rVert_{\ell^2_j H^1} \lesssim \lVert \phi \rVert_{H^1}
\end{equation*}
For $\Pin_{j,\gg} \phi$, we note that
\begin{equation*}\begin{split}
    \Pin_{j,\gg} =& F_{2^j}(x)F_{\geq 2^{-\theta j}}(D) F_{\gg 2^j}(|x|) + \{\text{similar}\}\\
                =& F_{2^j}(x) F_{\geq 2^{-\theta j}}(D) \sum_{m > j + 10} F_{2^m}(|x|) + \{\text{similar}\}
\end{split}\end{equation*}
where $\{\text{similar}\}$ denotes symmetric terms involving the projectors $F_{2^j}(-x)F_{\geq 2^{-\theta j}}(-D)$.  The supports of $F_{2^j}(\pm x)$ and $F_{2^m}(|x|)$ are separated by a distance $\sim 2^m$, so by~\Cref{lem:half-freq-proj-supp-lem} and Young's inequality,
\begin{equation*}\begin{split}
    \lVert \Pin_{j,\gg} \phi \rVert_{\ell^2_j H^1_x} \lesssim& \left\lVert \sum_{m > j + 10} 2^{j-m} \lVert F_{2^m}(|x|) \phi \rVert_{H^1} \right\rVert_{\ell^2_j}\\
    \lesssim& \lVert F_{2^j}(|x|) \phi \rVert_{\ell^2_j H^1_x}\\
    \lesssim& \lVert \phi \rVert_{H^1}
\end{split}\end{equation*}
Finally, for $\Pin_{j,\ll} \phi $, we further divide in frequency by writing
\begin{equation*}\begin{split}
    \Pin_{j,\ll} \phi  =&\Pin_j \sum_{m < j - 10}  F_{2^m}(|x|) \phi\\
    =&\Pin_j \sum_{m < j - 10} F_{\leq 2^{-\frac{m+2j}{3}}} (|D|)  F_{2^m}(|x|) \phi\\
    &+ \Pin_j  \sum_{m < j - 10}  F_{> 2^{-\frac{m+2j}{3}}} (|D|) F_{2^m}(|x|) \phi 
\end{split}\end{equation*}
Now, by Bernstein's inequality,
\begin{equation*}
    \lVert F_{2^m}(|x|) F_{\leq 2^{-\frac{m+2j}{3}}} (|D|) \rVert_{L^2 \to L^2} \lesssim 2^{\frac{m}{2}} 2^{-\frac{m+2j}{6}} = 2^{\frac{m-j}{3}}
\end{equation*}
so, by duality,
\begin{equation*}
    \lVert \Pin_j F_{\leq 2^{-\frac{m+2j}{3}}} (|D|)  F_{2^m}(|x|) \phi \rVert_{H^1} \lesssim 2^{\frac{m-j}{3}} \lVert F_{2^m}(|x|) \phi \rVert_{H^1}
\end{equation*}
On the other hand, a direct application of~\Cref{lem:half-freq-proj-supp-lem} with $N=1$ shows that
\begin{equation*}
    \lVert \Pin_j  F_{> 2^{-\frac{m+2j}{3}}} (|D|)  F_{2^m}(|x|) \phi \rVert_{H^1} \lesssim 2^{\frac{m-j}{3}} \lVert F_{2^m}(|x|) \phi \rVert_{H^1} 
\end{equation*}
Thus, by Young's inequality,
\begin{equation*}\begin{split}
    \lVert \Pin_{j,\ll} \phi \rVert_{\ell^2_j H^1_x} \lesssim& \left\lVert \sum_{m < j - 10} 2^{\frac{m-j}{3}} \lVert F_{2^m}(|x|) \phi \rVert_{H^1} \right\rVert_{\ell^2_j}\\
    \lesssim& \lVert \phi \rVert_{H^1}\qedhere
\end{split}
\end{equation*}
\end{proof}

\subsection{Decomposition of the weakly bound state}\label{sec:uwb-decom}

We now rewrite $\uwb$ in a way that will allow us to prove~\eqref{eqn:uwb-decomp}.  The definitions are somewhat involved, so we will first introduce the decomposition, and then prove that it has the required properties in the following subsections.

Recall that 
\begin{equation*}\begin{split}
    \uwb(t) =& e^{-it\Delta} F_{\leq t^\alpha}(|x|)e^{it\Delta}F_{< t^{-\delta}}(|D|) u(t) + e^{-it\Delta}F_{> t^\alpha}(|x|)e^{it\Delta}u(t)\\
            =:& \ulow + v
\end{split}\end{equation*}
From the definition of $\ulow$, we immediately see that
\begin{equation}\label{eqn:ulow-rem-type}
    \lVert \partial_x \ulow \rVert_{L^2} \lesssim (t^{-\alpha} + t^{-\delta}) \lVert u \rVert_{L^2}
\end{equation}
so $\ulow = o_{\dot{H}^1}(1)$.  Turning to $v$, the usual Duhamel representation for $u$ gives
\begin{equation}\label{eqn:v-forward-Duhamel}\begin{split}
    v(x,t) =& e^{-it\Delta} F_{> t^\alpha}(|x|) u_0 + i\int_0^t e^{-it\Delta} F_{> t^\alpha}(|x|) e^{is\Delta} V(x,s) u(x,s)\;ds
\end{split}\end{equation}
Since $\wlim_{t \to \infty} e^{it\Delta} v(t) = \wlim_{t \to \infty} F_{> t^\alpha}(|x|) e^{it\Delta} u(t)  = 0$, we also have a backward-in-time Duhamel representation
\begin{equation}\label{eqn:v-backward-Duhamel}\begin{split}
    v(x,t) =& -i\int_t^\infty e^{-it\Delta} F_{> s^\alpha}(|x|) e^{is\Delta} V(x,s) u(x,s)\;ds\\
    &-\int_t^\infty e^{-it\Delta} \partial_s F_{> s^\alpha}(|x|) e^{is\Delta} u(x,s)\;ds 
\end{split}\end{equation}
Here, as in~\eqref{eqn:weak-integral-convention}, convergence of the integrals is in the weak topology.  Now, let $J(t)$ be such that $2^{J(t)-1} < t^{\rho} \leq 2^{J(t)}$ with 
\begin{equation}
    \rho \theta < \alpha < \rho < 1/2,\qquad\qquad \rho > \frac{1}{\sigma}
\end{equation} and define
\begin{equation}\label{eqn:u-loc-def}\begin{split}
    \uloc(t) := \sum_{j=0}^{J(t)-1} \uloc^{(j)}(t)
\end{split}\end{equation}
with
\begin{equation}\begin{split}
    \uloc^{(j)}(t) =& \Plow_j v(t) +i \Pin_j \int_0^t e^{-i(t-s)\Delta} V(x,s) u(x,s)\;ds\\
                & -i \Pout_j \int_t^\infty e^{-i(t-s)\Delta} V(x,s) u(x,s)\;ds\\
                =:& \uloc^{(j,\textup{low})}(t) + \uloc^{(j,\textup{fwd})}(t) + \uloc^{(j,\textup{bwd})}(t)
\end{split}\end{equation}
and
\begin{equation}
    \urem := \urem^{(\textup{nr})}(t) + \sum_{j=0}^{J(t)-1} \urem^{(j)}(t) + \urem^{(\geq J(t))}(t)
\end{equation}
where $\urem^{(\textup{nr})}(t)$ contains terms which decay to $0$ in $\dot{H}^1$ without a rate:
\begin{equation*}\begin{split}
    \urem^{(\textup{nr})}(t) =& \cPin_{J(t)} e^{-it\Delta} F_{> t^\alpha}(|x|) u_0\\
    &- \cPout_{J(t)} \int_t^\infty e^{-it\Delta} \partial_s F_{> s^\alpha}(|x|) e^{is\Delta} u(x,s)\;ds\\
    =:& \urem^{(\textup{nr},1)}(t) + \urem^{(\textup{nr},2)}(t)
\end{split}\end{equation*}
The term $\urem^{(j)}$ contains the terms localized to $|x| \sim 2^j$:
\begin{equation*}\begin{split}
    \urem^{(j)} =& -i \Pin_j \int_0^t e^{-it\Delta} F_{\leq t^\alpha}(|x|) e^{is\Delta} V(x,s) u(x,s)\;ds\\
    &+ i \Pout_j \int_t^\infty e^{-it\Delta} F_{\leq s^\alpha}(|x|) e^{is\Delta} V(x,s) u(x,s)\;ds\\
    =:& \urem^{(j,\textup{fwd})}(t) + \urem^{(j,\textup{bwd})}(t)
\end{split}\end{equation*}
And $\urem^{(\geq J(t))}$ contains terms localized to $|x| \geq 2^{J(t)}$:
\begin{equation}\label{eqn:urem-ge-J-def}\begin{split}
    \urem^{(\geq J(t))}(t) =& \Plow_{\geq J(t)} v\\
    &+i \Pin_{\geq J(t)} \int_0^t e^{-i(t-s)\Delta} V(x,s) u(x,s)\;ds\\
    &-i \Pin_{\geq J(t)} \int_0^t e^{-it\Delta} F_{\leq t^\alpha}(|x|) e^{is\Delta} V(x,s) u(x,s)\;ds\\
    &- i \Pout_{_\geq J(t)} \int_t^\infty e^{-i(t-s)\Delta} V(x,s) u(x,s)\;ds\\
    &+ i \Pout_{_\geq J(t)} \int_t^\infty e^{-it\Delta} F_{\leq s^\alpha}(|x|) e^{is\Delta} V(x,s) u(x,s)\;ds\\
    =:& \urem^{(\geq J(t),\textup{low})}(t) +  \urem^{(\geq J(t),\textup{fwd},1)}(t) + \urem^{(\geq J(t),\textup{fwd},2)}(t)\\& + \urem^{(\geq J(t),\textup{bwd},1)}(t) + \urem^{(\geq J(t),\textup{bwd},2)}(t)
\end{split}\end{equation}
By~\eqref{eqn:v-forward-Duhamel} and~\eqref{eqn:v-backward-Duhamel}, we see that $\uloc + \urem = v$.  In the next two subsections, we will show that 
\begin{equation}\label{eqn:u-rem-bd}
    \lim_{t \to \infty} \lVert \partial_x \urem(t) \rVert_{L^2} = 0
\end{equation} 
and 
\begin{equation}\label{eqn:u-loc-bd}
    \sup_t \lVert \jBra{x}^{\theta} \partial_x \uloc(t) \rVert_{L^2} < \infty
\end{equation}
which together imply~\eqref{eqn:uloc-desired}.

\subsection{Bounds for \texorpdfstring{$\urem$}{u\_rem}}\label{sec:urem-bd}
We first prove the bound~\eqref{eqn:u-rem-bd} for the remainder term.  From the definition of $\urem^{(\textup{nr},1)}$,~\Cref{lem:in-out-sum-bdds} allows us to conclude that
\begin{equation*}
    \lim_{t \to \infty} \lVert \partial_x\urem^{(\textup{nr},1)} \rVert_{L^2} \lesssim  \lim_{t \to \infty} \lVert F_{> t^\alpha}(|x|) u_0 \rVert_{H^1} = 0
\end{equation*}
To control $\partial_x \urem^{(\textup{nr},2)}(t)$, we distribute the derivative and use the fact that 
\begin{equation*}
    -\partial_s F_{> s^\alpha} = \partial_s F_{\leq s^\alpha}
\end{equation*}
to write
\begin{subequations}\begin{align}
    \partial_x \urem^{(\textup{nr},2)}(t) 
    =& \cPout_{J(t)}\int_t^\infty e^{-it\Delta} \partial_s\partial_x F_{\leq s^\alpha}(|x|) e^{is\Delta} u(x,s)\;ds\label{eqn:urem-nr-2-space-loc}\\
    &+ \cPout_{J(t)}\int_t^\infty e^{-it\Delta} \partial_sF_{\leq s^\alpha}(|x|) e^{is\Delta} \partial_x u(x,s)\;ds\label{eqn:urem-nr-2-u}\\
    &+ \left[\partial_x, \cPout_{J(t)} \right]\int_t^\infty e^{-it\Delta} \partial_s F_{\leq s^\alpha}(|x|) e^{is\Delta} u(x,s)\;ds\label{eqn:urem-nr-2-symbol-comm}
\end{align}\end{subequations}
For~\eqref{eqn:urem-nr-2-space-loc}, we observe that
\begin{equation*}
    \lVert \partial_s \partial_x F_{\leq s^\alpha}(|x|) \rVert_{L^\infty} \lesssim s^{-1-\alpha}
\end{equation*}
so
\begin{equation*}\begin{split}
    \lVert \eqref{eqn:urem-nr-2-space-loc} \rVert_{L^2} \lesssim& \int_t^\infty s^{-1-\alpha} \lVert u(s) \rVert_{L^2}\;ds\\
    \lesssim& t^{-\alpha} \lVert u \rVert_{L^\infty_t H^1_x}\\
    =& o_{L^2_x}(1)
\end{split}
\end{equation*}
To control~\eqref{eqn:urem-nr-2-u}, we observe that by the reproducting formula
\begin{equation*}
    \cPout_{J(t)} = \cPout_{J(t)} F_{\gtrsim s^{-\rho\theta}}(|D|)
\end{equation*}
for $s \geq t$, we can rewrite~\eqref{eqn:urem-nr-2-u} with a time-dependent cut-off supported away from low frequencies:
\begin{equation*}
    \eqref{eqn:urem-nr-2-u} = \cPout_{J(t)}\int_t^\infty e^{-it\Delta} F_{\gtrsim s^{-\rho\theta}}(|D|) \partial_sF_{\leq s^\alpha}(|x|) e^{is\Delta} \partial_x u(x,s)\;ds
\end{equation*}
We can commute the space and frequency projectors to find that
\begin{equation*}\begin{split}
    \eqref{eqn:urem-nr-2-u} =& \cPout_{J(t)}\int_t^\infty e^{-it\Delta}  \partial_sF_{\leq s^\alpha}(|x|) F_{\gtrsim s^{-\rho\theta}}(|D|) e^{is\Delta} \partial_x u(x,s)\;ds\\
    &+ O(t^{\rho\theta - \alpha} \lVert u \rVert_{L^\infty_t H^1_x})
\end{split}\end{equation*}
and since we assume that $\rho\theta < \alpha$, the second term vanishes as $t \to \infty$.  To control the leading term, we compute a propagation estimate for the operator
\begin{equation*}
    B(s) = \partial_x F_{\gtrsim s^{-\rho\theta}}(|D|) F_{\leq s^\alpha}(|x|) F_{\gtrsim s^{-\rho\theta}}(|D|) \partial_x
\end{equation*}
against the function
\begin{equation*}
    \psi(s) = e^{is\Delta} u(s)
\end{equation*}
The operator $B(s)$ has almost the same form as the one considered in~\Cref{lem:pres-wave-op} (the additional derivatives can be easily handled using the $L^\infty_tH^1_x$ bound for $u$, and the decay of $V'$), so by modifying the proof there, we conclude that
\begin{equation*}
    \partial_s \langle B(s) \rangle_{s,\psi} \leq \left\langle \partial_x F_{\gtrsim s^{-\rho\theta}}(|D|) \partial_s F_{\leq s^\alpha}(|x|) F_{\gtrsim s^{-\rho\theta}}(|D|) \partial_x \right\rangle_{s,\psi} + O_{L^1_s}(\lVert u \rVert_{L^\infty_sH^1_x}^2)
\end{equation*}
and hence for $\psi(s) = e^{is\Delta} u(s)$,
\begin{equation*}
    \left\lVert \sqrt{\partial_s F_{\leq s^\alpha}(|x|)} F_{\gtrsim s^{-\rho\theta}}(|D|) \partial_x \psi \right\rVert_{L^2_{s,x}} \lesssim \lVert u \rVert_{L^\infty_sH^1_x} 
\end{equation*}
\Cref{lem:in-out-sum-bdds} and duality now imply that
\begin{equation*}\begin{split}
    \lVert \eqref{eqn:urem-nr-2-u} \rVert_{L^2} \lesssim& \left\lVert \int_t^\infty \partial_s F_{\leq s^\alpha}(|x|)  F_{\gtrsim s^{-\rho\theta}}(|D|) \partial_x \psi(s)\;ds\right\rVert_{L^2} + o_{t\to\infty}(\lVert u \rVert_{L^\infty_tH^1_x})\\
    \lesssim&\!\!\!\!\sup_{\lVert \phi \rVert_{L^2} = 1} \left\lVert \sqrt{\partial_s F_{\leq s^\alpha}(|x|)} \phi \right\rVert_{L^2_{s,x}} \left\lVert \sqrt{\partial_s F_{\leq s^\alpha}(|x|)} F_{\gtrsim s^{-\rho\theta}}(|D|) \partial_x  \psi \right\rVert_{L^2(t,\infty; L^2_x)}\\
    &+ o_{t\to\infty}(\lVert u \rVert_{L^\infty_tH^1_x})\\
    =& o_{t\to\infty}(\lVert u \rVert_{L^\infty_tH^1_x})
\end{split}\end{equation*}
For~\eqref{eqn:urem-nr-2-symbol-comm}, we see that the commutator term
\begin{equation*}
    [\partial_x, \cPout_{J(t)}] = [\partial_x, \Pout_{\geq J(t)}] + \sum_{j=0}^{J(t)-1} [\partial_x, \Pout_j]
\end{equation*}
has essentially the same form as the original projection $\cPout_{J(t)}$, but with better decay in $j$ in the sum.  Thus, we can obtain the bounds for~\eqref{eqn:urem-nr-2-symbol-comm} by defining a propagation observable
\[B(s) =  F_{\gtrsim s^{-\rho\theta}}(|D|) F_{\leq s^\alpha}(|x|) F_{\gtrsim s^{-\rho\theta}}(|D|)\]
and modifying the argument for~\eqref{eqn:urem-nr-2-u}.

Next, we show that $\urem^{(\geq J(t))}(t)$ goes to $0$ in $\dot{H}^1$.  Recalling the decomposition~\eqref{eqn:urem-ge-J-def}, we have that
\begin{equation*}\begin{split}
    \lVert \partial_x \urem^{(\geq J(t),\textup{low})} \rVert_{L^2} \lesssim&  \lVert [\partial_x ,F_{\geq 2^{J(t)}}(|x|)] \rVert_{L^2 \to L^2} \lVert u(t) \rVert_{L^2} + \lVert \partial_x F_{< 2^{-\theta J(t)}}(|D|) u(t) \rVert_{L^2}\\
    \lesssim& 2^{-\theta J(t)}\lVert u(t) \rVert_{L^2}\\
\end{split}\end{equation*}
which decays to $0$ as required.  

We now show how to bound the terms $\urem^{(\geq J(t), \textup{fwd}, 1)}$, $\urem^{(\geq J(t), \textup{bwd}, 1)}$, $\urem^{(\geq J(t), \textup{fwd}, 2)}$, and $\urem^{(\geq J(t), \textup{bwd}, 2)}$.  For $\urem^{(\geq J(t), \textup{fwd},2)}(t)$, we expand the definition of~$\Pin_{\geq J(t)}$ to write
\begin{align}
    \partial_x \urem^{(\geq J(t), \textup{fwd},2)} =& -i F_{\geq 2^{J(t)}}(x) \int_0^t\!\!\! e^{-it\Delta} F_{\geq 2^{-\theta J(t)}}(D) F_{\leq t^\alpha}(|x|) e^{is\Delta} \partial_x(V u)ds\label{eqn:urem-ge-J-fwd-2-1}\\
    &+ \{\textup{similar and easier terms}\}\notag
\end{align}
where $\{\textup{similar and easier terms}\}$ denotes terms containing the phase space cut-offs $F_{\geq 2^{J(t)}}(-x) F_{\geq 2^{-\theta J(t)}}(-D)$ (which can be handled by straightforward modifications of the methods we will outline below) and terms involving the commutator $[\partial_x, \Pin_{\geq J(t)}]$ (which have better time decay).  Define
\begin{equation*}
    K_t(x,y) = \frac{F_{\geq 2^{J(t)}}(x) F_{\leq t^\alpha}(|y|)}{2\pi} \int e^{i(t\xi^2 +\xi(x-y))} F_{\geq 2^{-\theta J(t)}}(\xi)\;d\xi
\end{equation*}
to be the integral kernel associated with $F_{\geq 2^{J(t)}}(x) e^{-it\Delta} F_{\geq 2^{-\theta J(t)}}(D) F_{\leq t^\alpha}(|x|)$.  Since $\rho > \alpha$, we have that $x-y \gtrsim t^{\rho}$ on the support of $K_t$.  In particular, this implies that the phase of the exponential is nonstationary, so we can repeatedly integrate by parts to obtain that
\begin{equation*}\begin{split}
    |K_t(x,y)| =& \left|\sum_{a+b = N}\!\! C_{a,b} \frac{F_{\geq 2^{J(t)}}(x) F_{\leq t^\alpha}(|y|)}{2\pi} \int e^{i(t\xi^2 +\xi(x-y))} \frac{t^a 2^{b \theta J(t)} F_{\geq 2^{-\theta J(t)}}^{(b)}(\xi)}{[2t\xi + (x-y)]^{N+a}}d\xi\right|\\
    \lesssim_N& F_{\geq 2^{J(t)}}(x) F_{\leq t^\alpha}(|y|) \frac{t^{(N-1)\rho\theta}}{[t^{1-\rho\theta } +(x-y)]^{N-1}}
\end{split}\end{equation*}
Thus, Schur's test gives us the estimate
\begin{equation*}
    \lVert F_{\geq 2^{J(t)}}(x) e^{-it\Delta} F_{\geq 2^{-\theta J(t)}}(D) F_{\leq t^\alpha}(|x|) \rVert_{L^2 \to L^2} \lesssim_N t^{\rho\theta} t^{-(1-2\rho\theta)(N-2)}
\end{equation*}
Choosing $N$ sufficiently large and recalling that $\rho\theta < 1/2$, we conclude that $\urem^{(\geq J(t), \textup{fwd}, 2)}$ vanishes in $\dot{H}^1$ as $t \to \infty$.

The argument for $\urem^{(\geq J(t), \textup{bwd}, 2)}$ is slightly more complicated, since the growth of the cut-off $F_{\leq s^{\alpha}}(|x|)$ appearing in the integral means that $\Pout_{\geq J(t)} e^{-it\Delta} F_{\leq s^{\alpha}}(|x|)$ does not project to classically forbidden trajectories.  To handle this term, we instead use the $L^2$ boundedness of $\Pout_{\geq J(t)}$ and $[\partial_x, \Pout_{\geq J(t)}]$ to estimate
\begin{equation*}\begin{split}
    \lVert \partial_x \urem^{(\geq J(t), \textup{bwd}, 2)}(t) \rVert_{L^2_x} \lesssim& \int_t^\infty \lVert F_{\geq 2^{-\theta J(t)}}(-D) F_{\leq s^\alpha}(|x|) e^{is\Delta} \partial_x(V u) \rVert_{L^2_x}\;ds\\
    &+ \int_t^\infty \lVert F_{\geq 2^{-\theta J(t)}}(-D) F_{\leq s^\alpha}(|x|) e^{is\Delta} V u \rVert_{L^2_x} \;ds\\
    &+ \{\textup{similar}\}
\end{split}\end{equation*}
where as before $\{\textup{similar}\}$ denotes terms with the opposite sign in the frequency variable.  In particular, it suffices to show that the $L^2 \to L^2$ operator norm of
\begin{equation*}
     A(t,s) = F_{\geq 2^{-\theta J(t)}}(-D)  F_{\leq s^\alpha}(|x|) e^{is\Delta} \jBra{x}^{-\sigma}
\end{equation*}
decays at an integrable rate in $s$.  To do this, we introduce the decomposition
\begin{equation*}\begin{split}
    A(t,s) =& F_{\geq 2^{-\theta J(t)}}(-D) F_{\leq s^\alpha}(|x|) e^{is\Delta} F_{\ll 2^{-\theta J(s)}}(|D|) \jBra{x}^{-\sigma}\\
    &+ F_{\geq 2^{-\theta J(t)}}(-D) F_{\leq s^\alpha}(|x|) e^{is\Delta} F_{\gtrsim 2^{-\theta J(s)}}(|D|) \jBra{x}^{-\sigma}\\
    =:& A_1(t,s) + A_2(t,s)
\end{split}\end{equation*}
Evidently, it is enough to obtain integrable decay for the two pieces.  For $A_2$, we immediately see from~\Cref{cor:loc-decay-weighted} that
\begin{equation*}\begin{split}
    \lVert A_2(t,s) \rVert_{L^2 \to L^2} \lesssim& \lVert F_{\leq s^\alpha}(|x|) e^{is\Delta} F_{\gtrsim 2^{-\theta J(s)}}(|D|) \jBra{x}^{-\sigma} \rVert_{L^2 \to L^2}\\
    \lesssim_\beta& s^{-\beta \sigma}
\end{split}\end{equation*}
for any $\beta < 1 -\theta \rho$.  Since $1 - \theta \rho > 1/2$, it follows that we can choose $\beta$ such that the second term is integrable.  Turning to $A_1$, we see that
\begin{equation*}\begin{split}
    \lVert A_1(t,s) \rVert_{L^2 \to L^2} \lesssim& \lVert F_{\geq 2^{-\theta J(t)}}(-D) F_{\leq s^\alpha}(|x|) F_{\ll 2^{-\theta J(s)}}(|D|) \rVert_{L^2 \to L^2}\\
    \lesssim_N& (s^{\alpha} t^{-\theta\rho})^{-N}\\
    \lesssim_N& s^{-N(\alpha -\theta \rho)}
\end{split}\end{equation*}
Since $\theta\rho < \alpha$, this is integrable in $s$ for $N$ large enough, completing the bound for $\urem^{(\geq J(t), \textup{bwd}, 2)}$.

Turning to $\urem^{(\geq J(t), \textup{fwd}, 1)}$ and taking $\beta$ such that
\begin{equation}\label{eqn:beta-def}
    \frac{1}{\sigma} < \beta < \rho
\end{equation}
we see that
\begin{equation*}\begin{split}
    \lVert \partial_x \urem^{(\geq J(t), \textup{fwd}, 1)} \rVert_{L^2}\!\! \lesssim&\left\lVert \int_0^t\! F_{\geq 2^{J(t)}}(x) F_{\geq 2^{-\theta J(t)}}(D) e^{-i(t-s)\Delta} F_{\leq t^\beta}(|x|) \partial_x(Vu)ds\right\rVert_{L^2}\\
    &+ \left\lVert \int_0^t\! F_{\geq 2^{J(t)}}(x) F_{\geq 2^{-\theta J(t)}}(D) e^{-i(t-s)\Delta} F_{\geq t^\beta}(|x|) \partial_x(Vu)ds\right\rVert_{L^2}\\
    &+ \{\textup{similar or easier terms}\}\\
    \lesssim& \int_0^t \lVert  F_{\geq 2^{J(t)}}(x) F_{\geq 2^{-\theta J(t)}}(D) e^{-i(t-s)\Delta} F_{\leq t^\beta}(|x|) \rVert_{L^2 \to L^2} \lVert u \rVert_{H^1_x} ds\\
    &+ t^{1-\beta\sigma} \lVert u \rVert_{L^\infty_t H^1_x}\\
    &+ \{\textup{similar or easier terms}\}\\
\end{split}\end{equation*}
so we only need to proving bounds for $F_{\geq 2^{J(t)}}(x) F_{\geq 2^{-\theta J(t)}}(D) e^{-i(t-s)\Delta} F_{\leq t^\beta}(|x|)$.  The kernel associated with this operator is
\begin{equation*}
    K_{t,s}(x,y) = \frac{F_{\geq 2^{J(t)}}(x) F_{\leq t^\beta}(|y|)}{2\pi} \int e^{i((t-s)\xi^2 + (x-y)\xi)} F_{\geq 2^{-\theta J(t)}}(\xi)\;d\xi
\end{equation*}
Repeated integration by parts gives the bound
\begin{equation*}\begin{split}
    |K_{t,s}(x,y)| \lesssim& \left|\sum_{a + b = N} C_{a,b} \int \frac{(t-s)^at^{b\rho\theta }}{(2(t-s)\xi + (x-y))^{N+a}} F^{(b)}_{\geq 2^{-\theta J(t)}}(\xi)\;d\xi\right|\\
        \lesssim_N& \frac{t^{(N-1) \rho\theta}}{(2(t-s)t^{-\rho\theta} + (x-y))^{N-1}}
\end{split}\end{equation*}
Since $K_{t,s}$ is supported in the set $|x-y| \gtrsim t^\rho$, we find that
\begin{equation*}
    \lVert K_{t,s} \rVert_{L^2_{x,y}} \lesssim_N t^{\rho\theta} t^{-(N-2)(1-\theta)\rho}
\end{equation*}
Taking $N$ large shows that this term vanishes as $t \to \infty$.  Furthermore, the bounds for $\urem^{(\geq J(t), \textup{bwd}, 1)}$ follow by the same logic, which completes the argument for $\urem^{(\geq J(t))}$.

It only remains to bound the sum of the $\urem^{(j)}$.  We will focus on $\urem^{(j,\textup{fwd})}$, since the bounds for the backward in time terms are similar.  By~\Cref{lem:space-freq-loc-decay}, we see that for $2^j < t^\rho$
\begin{equation*}
    \lVert \Pin_j e^{-it\Delta} F_{\leq t^\alpha}(|x|) \rVert_{L^2 \to L^2} \lesssim_N t^{-N}
\end{equation*}
since $\theta \rho < \min(1/2, 1-\rho, 1-\beta)$ by the assumption that $\rho < 1/2$.  A similar bound holds for the commutator of this operator with $\partial_x$.  Since $V(x,s) u(x,s) \in L^\infty_s H^1_x$, this implies that
\begin{equation*}\begin{split}
    \lVert \partial_x \urem^{(j,\textup{fwd})}(t) \rVert_{L^2} =& \left\lVert \partial_x \int_0^t \Pin_j e^{-it\Delta} F_{\leq t^\alpha}(|x|) e^{is\Delta} Vu(s)\;ds \right\rVert_{L^2}\\
    \lesssim& t^{1-N} \lVert u \rVert_{L^\infty_tH^1}
\end{split}\end{equation*}
Based on our definition of $J(t)$, $J(t) \leq 1 + \rho \log_2\jBra{t}$, so
\begin{equation*}
    \sum_{j=0}^{J(t)-1} \lVert \partial_x \urem^{(j,\textup{fwd})}(t) \rVert_{L^2} \lesssim (1 + \rho \log_2\jBra{t} ) t^{1-N}\lVert u \rVert_{L^\infty_tH^1} = o(\lVert u \rVert_{L^\infty_tH^1})
\end{equation*}
as required.  Meanwhile, the bounds for the sum of the $\partial_x\urem^{(j, \textup{bwd})}$ terms follow from a straightforward modification of the argument used for $\urem^{(\geq J(t),\textup{bwd},2)}$, completing the argument for $\partial_x\sum_{j=0}^{J(t)-1}\urem^{(j)}$.

\subsection{Bounds for \texorpdfstring{$\uloc$}{u\_loc}}\label{sec:u-loc-bd-subsec}

We now show how to prove~\eqref{eqn:u-loc-bd}.  For the low-frequency part $\uloc^{(j,\textup{low})}$, we note that 
\begin{equation*}\begin{split}
    \lVert \jBra{x}^\theta \partial_x \uloc^{(j,\textup{low})} \rVert_{L^2}^2 \lesssim& 2^{2(\theta - 1) j} \lVert F'_{2^j}(|x|) F_{< 2^{-\theta j}}(|D|) v \rVert_{L^2}^2\\
    &+ 2^{2\theta j} \lVert F_{2^j}(|x|) F_{< 2^{-\theta j}}(|D|) \partial_x v \rVert_{L^2}^2
\end{split}\end{equation*}
with the obvious modifications for $j = 0$.  For the first term, we can sum directly to obtain
\begin{equation*}
    \sum_{j \geq 0} 2^{2(\theta - 1) j} \lVert F'_{2^j}(|x|) F_{< 2^{-\theta j}}(|D|) v \rVert_{L^2}^2 \lesssim \lVert v \rVert_{L^2}^2
\end{equation*}
Turning to the second term, we observe that by the proof of~\Cref{cor:full-freq-proj-supp-lem},
\begin{equation*}
    \lVert F_{2^j}(|x|) F_{< 2^{-\theta j}}(|D|) \partial_x \left((1-F_{\sim 2^j}(|x|)) v\right)\rVert_{L^\infty_t L^2}\lesssim_N 2^{-\theta j} 2^{-N(1-\theta)j} \lVert v \rVert_{L^\infty_t L^2_x}
\end{equation*}
so
\begin{equation*}\begin{split}
     \sum_{j \geq 0} 2^{2\theta j} \lVert F_{2^j}(|x|) F_{< 2^{-\theta j}}(|D|) \partial_x v \rVert_{L^2}^2 \lesssim& \sum_{j \geq 0} 2^{2\theta j} \left\lVert F_{< 2^{-\theta j}}(|D|) \partial_x \left(F_{\sim 2^j}(|x|) v\right) \right\rVert_{L^2}^2 \\
     +& \sum_{j \geq 0} 2^{-2N(1-\theta)j} \lVert v \rVert_{L^2}^2\\
     \lesssim& \lVert v \rVert_{L^2}^2
\end{split}\end{equation*}
Combining these estimates, we see that
\begin{equation*}
    \left\lVert \sum_{j=0}^\infty \jBra{x}^\theta \partial_x \uloc^{(j,\textup{low})} \right\rVert_{L^2}^2 \lesssim \lVert v \rVert_{H^1}^2 < \infty
\end{equation*}
Thus, we only need to control $\uloc^{(j,\textup{fwd})}$ and $\uloc^{(j,\textup{bwd})}$.  The required bounds follow immediately from~\Cref{lem:in-out-lem} with $S' = 1, S = 0$:  For $\jBra{x}^\theta \partial_x \uloc^{(j,\textup{fwd})}$, we find that
\begin{equation*}\begin{split}
    \sum_{j=0}^\infty \lVert \jBra{x}^\theta\partial_x \uloc^{(j,\textup{fwd})}\rVert_{L^2} =& \sum_{j=0}^\infty \left\lVert \jBra{x}^\theta \partial_x \Pin_j \int_0^t e^{-i(t-s)\Delta} V(x,s) u(x,s) \;ds\right\rVert_{L^2}\\
    \lesssim& \lVert \jBra{x}^\sigma Vu \rVert_{L^\infty_t H^1_x} \lesssim \lVert u \rVert_{L^\infty_t H^1_x}
\end{split}\end{equation*}
and a similar bound holds for $\jBra{x}^\theta \partial_x \uloc^{(j,\textup{bwd})}$, which is sufficient to get~\eqref{eqn:u-loc-bd}.

\section{Higher-order symbol bounds}


We now show how to prove the higher-order symbol bounds~\eqref{eqn:uloc-higher-order}.  Intuitively, we would like to proceed by induction: Having proven that the localized part satisfies~\eqref{eqn:uloc-higher-order} for some $n$, we want to use the symbol bounds for $V$ together with~\Cref{lem:in-out-lem} to prove better decay for higher derivatives.  The problem is that the Duhamel integral for $\uloc^{(j)}$ involves $u$, while the decay bounds were only proved for $\uloc$, preventing us from iterating directly.

\subsection{Definition of $u_{\textup{loc},n}$}
In order to obtain symbol bounds, we must redefine the localized part of the solution.  Let us write
\begin{equation}
    u_{\textup{loc},1} := \uloc
\end{equation}
and define for $n > 1$
\begin{equation}\label{eqn:uloc-n-def}
    u_{\textup{loc},n} := \sum_{j = 0} ^{J(t)-1} u_{\textup{loc},n}^{(j)}
\end{equation}
with
\begin{equation}\label{eqn:uloc-n-j-def}\begin{split}
    u_{\textup{loc},n}^{(j)} =& \Plow_j v(t) + i \Pin_j \int_0^t e^{-i(t-s)\Delta} V u_{\textup{loc},n-1}(s)\;ds\\
    &- i \Pout_j \int_t^\infty e^{-i(t-s)\Delta} Vu_{\textup{loc},n-1}(s)\;ds\\
    =:& u_{\textup{loc},n}^{(j,\textup{low})} + u_{\textup{loc},n}^{(j,\textup{fwd})} + u_{\textup{loc},n}^{(j,\textup{bwd})}
\end{split}\end{equation}
We also define
\begin{equation}\label{urem-n-def}
    u_{\textup{rem},n}(t) = \urem^{(\textup{nr})}(t) + \urem^{(\geq J(t))}(t) + \sum_{j = 0}^{J(t)-1} u_{\textup{rem},n}^{(j)}(t)
\end{equation}
with
\begin{equation}\label{urem-n-j-def}\begin{split}
    u_{\textup{rem},n}^{(j)}(t) =& \urem^{(j)}(t) + i \Pin_j \int_0^t e^{-i(t-s)\Delta} V (u - u_{\textup{loc},n-1})(s)\;ds\\
    &- i \Pout_j \int_t^\infty e^{-i(t-s)\Delta} V (u - u_{\textup{loc},n-1})(s)\;ds\\
    =:& \urem^{(j)}(t) + u_{\textup{rem},n}^{(j,\textup{fwd})}(t) + u_{\textup{rem},n}^{(j,\textup{bwd})}(t)
\end{split}\end{equation}
As before, we have the identity
\begin{equation*}
    v(t) = u_{\textup{rem},n}(t) + u_{\textup{loc},n}(t)
\end{equation*}
so to complete the proof of~\Cref{thm:main-thm}, it suffices to show that
\begin{equation}\label{eqn:higher-order-u-loc-bd}
    \lVert \jBra{x}^{\theta k}\partial_x^k u_{\textup{loc},n} \rVert_{L^\infty_t L^2_x} < \infty,\qquad \textup{for } k = 0,1,2, \cdots n
\end{equation}
and that
\begin{equation}\label{eqn:higher-order-rem-cond}
    \lim_{t \to \infty} \lVert \partial_x u_{\textup{rem},n}(t) \rVert_{L^2} = 0
\end{equation}
The results of~\Cref{sec:energy-loc} state that~\eqref{eqn:higher-order-u-loc-bd} and~\eqref{eqn:higher-order-rem-cond} are true when $n = 1$.  To prove that they hold for $n > 1$, we proceed by induction and suppose that they hold for some $n \geq 1$, and prove that this implies the corresponding bounds for $n+1$.  We will prove these statements in the following two subsections.  The proof of~\eqref{eqn:higher-order-u-loc-bd} is given in~\Cref{sec:higher-order-loc-bds}, and amounts to an application of~\Cref{lem:in-out-lem} together with the strengthened inductive hypothesis
\begin{equation}\label{eqn:higher-order-loc-bd-strong}
    \lVert \jBra{x}^{\theta k}\partial_x^k u_{\textup{loc},n} \rVert_{L^\infty_t H^S_x} < \infty,\qquad \textup{for } k = 0,1,2, \cdots n,\;\; 0 \leq s < 1
\end{equation}
From here, it only remains to prove~\eqref{eqn:higher-order-rem-cond}.  We will say that a function $w(x,t)$ is \textit{remainder-type} if $\lim_{t \to \infty} \lVert \partial_x w \rVert_{L^2} = 0$.  Then,~\eqref{eqn:higher-order-rem-cond} is simply the statement that $u_{\textup{rem},n+1}$ is remainder-type.  To prove this, it suffices to prove that $\sum_{j=0}^{J(t)-1} u_{\textup{rem},n+1}^{(j,\textup{fwd})}$ and $\sum_{j=0}^{J(t)-1} u_{\textup{rem},n+1}^{(j,\textup{bwd})}$ are both remainder-type.  To help with this, let us define the projected Duhamel integral operators
\begin{equation}\label{eqn:proj-Duhamel-def}\begin{split}
    \mathcal{D}^{(j,\textup{fwd})} \phi(t) := i\Pin_j \int_0^t e^{-i(t-s)\Delta} V \phi(s)\;ds \\
    \mathcal{D}^{(j,\textup{bwd})} \phi(t) := -i \Pout_j \int_t^\infty e^{-i(t-s)\Delta} V \phi(s)\;ds \\
\end{split}
\end{equation}
so 
$$u_{\textup{rem},n+1}^{(j,\textup{fwd})} = \mathcal{D}^{(j,\textup{fwd})} (u - u_{\textup{loc},n})$$
and similarly for $u_{\textup{rem},n+1}^{(j,\textup{bwd})}$.  A quick calculation shows that
$$u(t) - u_{\textup{loc},n}(t) = u_{\textup{lin}}(t) + (\Jfree(t) u(t) - u_{\textup{lin}}(t)) + \ulow(t) + u_{\textup{rem},n}(t)$$
where $u_{\textup{lin}}(t) = e^{-it\Delta} \Omega_\text{free} u_0$ is the linearly scattering part of the solution.  Based on the inductive hypothesis,~\eqref{eqn:higher-order-rem-cond} will follow once we prove that $\sum_{j=0}^{J(t)-1} \mathcal{D}^{(j,\textup{fwd})} w$ and $\sum_{j=0}^{J(t)-1} \mathcal{D}^{(j,\textup{bwd})} w$ are remainder-type whenever $w$ is remainder-type or $w = e^{-it\Delta} \phi$ for some time-independent function $\phi$, since $\Jfree(t) u(t) - u_{\textup{lin}}(t)$ is remainder-type by the definition of $u_+$, $\ulow$ is remainder-type by~\eqref{eqn:ulow-rem-type}, and the inductive hypothesis states that $u_{\textup{rem},n}(t)$ is remainder-type.  We carry out this proof in~\Cref{sec:higher-order-rem-decay}.

\subsection{Symbol bounds for $u_{\textup{loc},n+1}$}\label{sec:higher-order-loc-bds}

Our goal in this section is to prove that~\eqref{eqn:higher-order-loc-bd-strong} holds for all $n$.  For $n = 1$, this follows essentially from the same steps used in~\Cref{sec:u-loc-bd-subsec}: Because of the low-frequency projection, $\lVert \jBra{x}^\theta \partial_x \uloc^{(j, \textup{low})} \rVert_{L^\infty_t\dot{H}^1_x}$ has better decay (and hence better summability properties) than $\lVert \jBra{x}^\theta \partial_x \uloc^{(j, \textup{low})} \rVert_{L^\infty_tL^2_x}$, while~\Cref{lem:in-out-lem} still suffices to control $\uloc^{(j,\textup{fwd})}$ and $\uloc^{(j,\textup{bwd})}$ in $H^s$ ($0 \leq s < 1$) since $u$ (and thus $Vu$) are in $H^1$.  

We will now argue inductively to show that if~\eqref{eqn:higher-order-loc-bd-strong} holds for some $n$, then it also holds with $n$ replaced by $n+1$.  Recalling the definitions~\cref{eqn:uloc-n-def,eqn:uloc-n-j-def} and arguing as in~\Cref{sec:u-loc-bd-subsec}, we find that
\begin{equation*}\begin{split}
    \left\lVert \jBra{x}^{\theta k}\partial_x^k u_{\textup{loc},n+1}^{(j, \textup{low})} \right\rVert_{L^\infty_t L^2_x}^2 \lesssim& 2^{2k(\theta -1)j} \lVert F^{(k)}_{2^j}(|x|) F_{< 2^{-\theta j}}(|D|) v \rVert_{L^2}^2\\
    &+ 2^{2k\theta j} \lVert F_{2^j}(|x|) F_{< 2^{-\theta j}}(|D|) \partial_x^k v \rVert_{L^2}^2\\
    &+ \{\textup{similar}\}
\end{split}\end{equation*}
is summable in $j$, which is compatible with~\eqref{eqn:higher-order-loc-bd-strong} for $s = 0$.  Since $u_{\textup{loc},n+1}^{(j, \textup{low})}$ is supported on low frequencies, the argument for $s > 0$ is simpler.  

It only remains to bound $u_{\textup{loc},n+1}^{(j, \textup{fwd})}$ and $u_{\textup{loc},n+1}^{(j, \textup{bwd})}$. The two bounds are similar, so we will only give the argument for the forward-in-time terms $u_{\textup{loc},n+1}^{(j, \textup{fwd})}$.  Observe that
\begin{equation*}\begin{split}
    \lVert \jBra{x}^{k\theta} \partial_x^k u_{\textup{loc},n+1}^{(j, \textup{fwd})}(t) \rVert_{L^\infty_t H^{S}_x} =& \left\lVert \jBra{x}^{k\theta} \partial_x^k \Pin_j \int_0^t e^{-i(t-s)\Delta} V u_{\textup{loc},n}(s)\;ds \right\rVert_{L^\infty_t H^{S'}_x}
\end{split}\end{equation*}
with $S < S' < 1$.  Applying~\Cref{lem:in-out-lem}, we find that
\begin{equation*}\begin{split}
    \sum_{j=0}^\infty \lVert \jBra{x}^{k\theta} \partial_x^k u_{\textup{loc},n+1}^{(j, \textup{fwd})}(t) \rVert_{L^\infty_t H^{S}_x} \lesssim&  \sum_{m=0}^{k-1} \lVert \jBra{x}^{m\theta} \partial_x^m u_{\textup{loc},n} \rVert_{L^\infty_t H^{S'}_x}\\
    <& \infty
\end{split}\end{equation*}  
Based on the symbol-type bounds~\eqref{eqn:V-symb-bds} for $V$ and the fact that $u_{\textup{loc},n}$ satisfies~\eqref{eqn:higher-order-loc-bd-strong}, we conclude that $u_{\textup{loc},n+1}^{(j, \textup{fwd})}(t)$ satisfies bounds compatible with~\eqref{eqn:higher-order-loc-bd-strong}.  Since $u_{\textup{loc},n+1}^{(j, \textup{bwd})}(t)$ obeys similar estimates, this concludes the proof of~\eqref{eqn:higher-order-loc-bd-strong} (and thus also of~\eqref{eqn:higher-order-u-loc-bd}).

\subsection{Proof that $u_{\textup{rem},n}$ is remainder-type}\label{sec:higher-order-rem-decay} 

We will prove~\Cref{eqn:higher-order-rem-cond} as a consequence of the following result:
\begin{lemma}\label{lem:decay}
    Suppose that $\phi \in L^\infty_t H^1_x$ has the property that for any compact set $K \subset \bbR$,
    \begin{equation}\label{eqn:decay-lemma-hypo}
        \lim_{t \to \infty} \lVert \phi(t) \rVert_{L^2(K)} + \lVert \partial_x \phi(t) \rVert_{L^2(K)} = 0
    \end{equation}
    Then,
    \begin{equation}\label{eqn:decay-lemma-conclusion}
        \lim_{t \to \infty} \sum_{j=0}^\infty \left(\left\lVert \partial_x \Djfwd \phi \right\rVert_{L^2} + \left\lVert \partial_x \Djbwd \phi \right\rVert_{L^2} \right) = 0
    \end{equation}
\end{lemma}
\begin{proof}
    We will begin with the proof for $j = 0$.  Based on the definition of $\Pin_0$ and the improved decay of $\partial_x V$, we have that
    \begin{subequations}\begin{align}
        \left\lVert \partial_x \Dzfwd \phi \right\rVert_{L^2} \leq& \left\lVert \Pin_0 \int_0^t e^{-i(t-s)\Delta} V \partial_x \phi(s)\;ds \right\rVert_{L^2}\notag\\
                                                                &+ \{\textup{similar or easier}\}\notag\\
                                                              \leq& \sum_{k \geq -1}\left\lVert \Pin_0 P_k \int_0^t e^{-i(t-s)\Delta} F_{\leq 1}(|x|) V \partial_x\phi(s)\;ds \right\rVert_{L^2}\label{eqn:Dzfwd-ell0}\\
                                                                  &+\sum_{\substack{k \geq -1\\ \ell > 0}} \left\lVert \Pin_0 P_k \int_0^t e^{-i(t-s)\Delta} F_{2^\ell}(|x|)V \partial_x \phi(s)\;ds \right\rVert_{L^2} \label{eqn:Dzfwd-ell}\\
                                                                  &+ \{\textup{similar or easier}\}\notag
    \end{align}\end{subequations}
    We will focus on controlling the second sum~\eqref{eqn:Dzfwd-ell}; the estimates for~\eqref{eqn:Dzfwd-ell0} are similar.  Let us further decompose~\eqref{eqn:Dzfwd-ell} as
    \begin{equation*}
        \eqref{eqn:Dzfwd-ell} = \sum_{\substack{k \geq -1\\ \ell > 0}} A_{\ell,k}(t) + B_{\ell,k}(t)
    \end{equation*}
    with
    \begin{equation*}\begin{split}
        A_{\ell,k}(t) =& \left\lVert \Pin_0 P_k \int_0^{t - C2^{\ell-k}} e^{-i(t-s)\Delta} F_{2^\ell}(|x|) V \partial_x\phi\;ds \right\rVert_{L^2_x}\\
        B_{\ell,k}(t) =& \left\lVert \Pin_0 P_k \int_{t - C2^{\ell-k}}^t e^{-i(t-s)\Delta} F_{2^\ell}(|x|) V \partial_x\phi\;ds \right\rVert_{L^2_x}
    \end{split}\end{equation*}
    For $B_{\ell,k}$, we see immediately that
    $$\sup_t|B_{\ell,k}(t)| \lesssim 2^{\ell-k} \lVert F_{2^\ell}(|x|) V \partial_x \phi \rVert_{L^\infty_t L^2_x} \lesssim 2^{(1-\sigma)\ell} 2^{-k} \lVert \partial_x \phi \rVert_{L^\infty_t L^2_x}$$
    which is summable in $\ell$ and $k$.  On the other hand, by~\eqref{eqn:decay-lemma-hypo} $\lVert F_{2^\ell}(|x|) \partial_x \phi(t) \rVert_{L^2} \to 0$ as $t \to \infty$, so for fixed $\ell$ and $k$ we have that
    $$\lim_{t \to \infty} |B_{\ell,k}(t)| \lesssim 2^{\ell-k} \lim_{t \to \infty} \lVert F_{2^\ell}(|x|) \partial_x \phi \rVert_{L^\infty(t-C2^{\ell-k}, t; L^2_x)} = 0$$
    Applying the dominated convergence theorem, we now conclude that
    $$\lim_{t \to \infty} \sum_{\substack{\ell > 0\\k \geq -1}} B_{\ell,k}(t) = 0$$
    For $A_{\ell,k}$, we observe that the kernel for $\Pin_0 P_k e^{-i\tau \Delta} F_{2^\ell}(|x|)$ is
    \begin{equation}\label{eqn:A-ell-k-kernel}
        \frac{F_{\leq 1}(|x|) F_{2^\ell}(|y|)}{2\pi} \int e^{i\tau\xi^2} e^{i\xi(x-y)} F_{\geq 1}(|\xi|) F_{2^k}(|\xi|)\;d\xi
    \end{equation}
    Based on the localization of $x$ and $y$, we see that the phase in the integral will be nonstationary for $\tau \gg 2^{\ell-k}$, so repeated integration by parts gives the estimate
    \begin{equation*}\begin{split}
        \lVert \Pin_0 P_k e^{-i\tau\Delta} F_{2^\ell}(|x|) \rVert_{L^2 \to L^2} \lesssim& \sum_{a+b=N} \sup_{|x|\lesssim 1,|y| \sim 2^\ell} \int \frac{|\tau|^a2^{\ell/2-bk}}{|2\tau\xi + (x-y)|^{N+a}} F_{2^k}^{(b)}(|\xi|)\;d\xi\\
        \lesssim& 2^{\ell/2} 2^{(1-2N)k} \tau^{-N}
    \end{split}\end{equation*}
    Thus,
    \begin{equation*}
        \sup_t |A_{\ell, k}(t)| \lesssim 2^{\ell/2} 2^{(1-2N)k} 2^{(1-N)(\ell-k)} \lVert \partial_x \phi \rVert_{L^\infty_t L^2}
    \end{equation*}
    is summable in $\ell$ and $k$.  Dividing the integral into $0 \leq s \leq t/2$ and $s \geq t/2$, we have that
    \begin{equation*}\begin{split}
        A_{\ell,k} \leq& \left\lVert \Pin_0 P_k \int_0^{t/2} e^{-i(t-s)\Delta} F_{2^\ell}(|x|) V \partial_x\phi\;ds \right\rVert_{L^2}\\
        &+ \left\lVert \Pin_0 P_k \int_{t/2}^{t- C 2^{\ell-k}} e^{-i(t-s)\Delta} F_{2^\ell}(|x|) V \partial_x\phi\;ds \right\rVert_{L^2}\\
        \lesssim& 2^{\ell/2} 2^{(1-2N)k} t^{1-N}\lVert \partial_x \phi \rVert_{L^\infty_t L^2_x} \\
        &+ 2^{\ell/2} 2^{(1-2N)k} 2^{(1-N)(\ell-k)} \lVert F_{2^\ell}(|x|) \partial_x \phi(s) \rVert_{L^\infty_s(t/2, t; L^2_x)}
    \end{split}\end{equation*}
    Taking limits, we see that $A_{\ell,k}(t) \to 0$ for each $\ell$ and $k$, so 
    $$\lim_{t \to \infty} \sum_{\substack{\ell > 0\\ k \geq -1}} A_{\ell,k} = 0$$
    In particular, this shows that
    $$\lim_{t \to \infty} \lVert \partial_x \Dzfwd \phi \rVert_{L^2} = 0$$
    
    For $j > 0$, we write
    \begin{equation*}\begin{split}
        \lVert \partial_x \Djfwd \phi \rVert_{L^2} \leq& \sum_{\substack{\ell > j - 10\\ k \geq -\theta j - 10}} \underbrace{\left\lVert \Pin_j P_k \int_0^{t-C2^{\ell-k}} e^{-i(t-s)\Delta} F_{2^\ell}(|x|) V \partial_x \phi \;ds \right\rVert_{L^2}}_{A^j_{\ell,k}}\\
        &+ \sum_{\substack{\ell > j - 10\\ k \geq -\theta j - 10}} \underbrace{\left\lVert \Pin_j P_k \int_{t-C2^{\ell-k}}^t e^{-i(t-s)\Delta} F_{2^\ell}(|x|) V \partial_x \phi \;ds \right\rVert_{L^2}}_{B^j_{\ell,k}}\\
        &+ \sum_{k \geq -\theta j - 10} \underbrace{\left\lVert \Pin_j P_k \int_0^t e^{-i(t-s)\Delta} F_{\leq 2^{j -10}}(|x|) V \partial_x \phi \;ds \right\rVert_{L^2}}_{C^j_k}\\
        &+ \{\textup{similar or easier}\}\\
    \end{split}
    \end{equation*}
    where $\{\textup{similar or easier}\}$ denotes terms where the derivative either falls on $V$ or commutes with the projection $\Pin_j$.  As before, we see that
    \begin{equation*}
        |B^j_{\ell,k}(t)| \lesssim 2^{(1-\sigma)\ell} 2^{-k} \lVert \partial_x \phi \rVert_{L^\infty_t L^2_x}
    \end{equation*}
    so,
    \begin{equation*}
        \sum_{\substack{\ell > j - 10\\ k \geq -\theta j - 10}} |B^j_{\ell,k}(t)| \lesssim 2^{(1+\theta - \sigma)j} \lVert \partial_x \phi \rVert_{L^\infty_t L^2_x} \in \ell^1_j
    \end{equation*}
    and, since $ B^j_{\ell,k}(t) \to 0$ for each fixed $j,k,\ell$,
    \begin{equation*}
        \lim_{t \to \infty} \sum_{j = 1}^\infty\sum_{\substack{\ell > j - 10\\ k \geq -\theta j - 10}} |B^j_{\ell,k}(t)| = 0
    \end{equation*}
    For $A^j_{\ell,k}$, a slight modification of the argument above shows that
    \begin{equation*}
        \lVert \Pin_j P_k e^{-i\tau \Delta} F_{2^\ell}(|x|) \rVert_{L^2 \to L^2} \lesssim 2^{j/2} 2^{\ell/2} 2^{(1-2N)k} \tau^{-N}
    \end{equation*}
    so
    \begin{equation*}
        A^j_{\ell,k}(t) \lesssim 2^{j/2} 2^{\ell/2} 2^{(1-2N)k} 2^{(1-N)(\ell-k)} \lVert \partial_x \phi \rVert_{L^2}
    \end{equation*}
    If we take $N > \frac{2}{1-\theta}$, then the right-hand side is summable in $j,k,$ and $\ell$ for $\ell > j - 10$, $k \geq -\theta j - 10$, and splitting the time integral into $0 \leq s \leq t/2$ and $s \geq t/2$ shows that $A^j_{\ell,k}(t) \to 0$ as $t \to \infty$, so this term is also acceptable.  Finally, for $C^j_k$, we observe that the kernel for $\Pin_j P_k e^{-i\tau \Delta} F_{\leq 2^{j-10}}(|x|)$ is given by
    \begin{equation*}
        K^j_k(x,y) = \frac{F_{2^j}(x) F_{\leq 2^{j-10}}(|y|)}{2\pi} \int e^{i\tau \xi^2} e^{i\xi(x-y)} F_{2^k}(|\xi|) F_{\geq 2^{-\theta j}}(\xi)\;d\xi + \{\textup{similar}\}
    \end{equation*}
    (where, as usual, $\{\textup{similar}\}$ denotes a term with projected to $-x \sim 2^j$, $-\xi \gtrsim 2^{-\theta j}$).  For $\tau \geq 0$ and $(x,y)$ in the support of $K^j_k$, the phase is nonstationary, and repeated integration by parts gives
    \begin{equation*}
        \lVert K^j_k \rVert_{L^2_{x,y}} \lesssim 2^{j+k} \frac{2^{-Nk}}{(2^k \tau + 2^j)^N} 
    \end{equation*}
    Thus, integrating shows that
    \begin{equation*}
        C^j_k \lesssim 2^{(1-N)(j+k)} 2^{j-k} \lVert \partial_x \phi \rVert_{L^\infty_t L^2_x}
    \end{equation*}
    which is summable in $j$ and $k \geq -\theta j - 10$ when $N > \frac{2}{1-\theta}$.  Moreover, dividing the integral into $0 \leq s \leq t/2$ and $t/2 \leq s \leq t$ and using~\eqref{eqn:decay-lemma-hypo} again shows that
    $$\lim_{t \to \infty}C^j_k(t) = 0$$
    This proves~\eqref{eqn:decay-lemma-conclusion} for $\Djfwd \phi$.  The argument for $\Djbwd \phi$ is analogous.
\end{proof}

Thus, it only remains to prove that $u_\text{lin}$ and $u_{\text{rem}, n}$ satisfy the hypotheses of~\Cref{lem:decay}.  The $L^\infty_t H^1$ bounds for $\ulin,$ $\ulow$, and $\Jfree u - \ulin$ follow immediately from the $L^\infty_tH^1_x$ bound for $u$ and the fact that $u_+ \in H^1$.  For $u_{\text{rem}, n}$, we will proceed by induction.  For $n=1$, we have that
\begin{equation*}
    u_{\text{rem}, 1} = v - \uloc
\end{equation*}
Recall that $\uloc = \sum_{j=0}^{J(t)-1} \left(\Plow_j v(t) +  \uloc^{(j, \text{fwd})} + \uloc^{(j, \text{bwd})}\right)$.  Now, since $\theta < 1$,~\Cref{cor:full-freq-proj-supp-lem} implies that
\begin{equation*}
    \lVert \Plow_j f \rVert_{L^2} = \lVert \Plow_j F_{\sim 2^j}(|x|) f \rVert_{L^2} + O(2^{-Nj} \lVert f \rVert_{L^2})
\end{equation*}
so
\begin{equation*}
    \sup_t \left\lVert \sum_{j = 0}^{J(t)-1}\Plow_j v(t) \right\rVert_{L^2_x}^2 \lesssim \sup_t \sum_{j=0}^\infty \left(\lVert F_{\sim 2^j}(|x|) v \rVert_{L^2}^2 + 2^{-2Nj} \lVert v \rVert_{L^2}^2\right) \lesssim \lVert u \rVert_{L^\infty_tL^2_x}^2
\end{equation*}
is bounded.  Similarly, recalling that
\begin{equation*}\begin{split}
    \uloc^{(j,\textup{fwd})}(t) =& \Pin_j (u(t) - e^{-it\Delta} u_0)\\
    \uloc^{(j,\textup{bwd})}(t) =& \Pout_j (u(t) - u_\textup{lin}(t)) 
\end{split}\end{equation*}
and using~\Cref{lem:in-out-sum-bdds}, we see that $\uloc \in L^\infty_t H^1_x$.  Since $v = e^{-it\Delta} F_{> t^\alpha}(|x|) e^{it\Delta} u \in L^\infty_tH^1_x$ as well, we conclude that $u_{\textup{rem}, 1} \in L^\infty_t H^1_x$.  

For $n > 1$, examining the proof of~\Cref{lem:decay} shows that $\partial_x u_{\textup{rem}, n} \in L^\infty_t L^2_x$ provided $u_{\text{rem},n-1} \in L^\infty_t H^1_x$.  Moreover, by removing the derivative and repeating the argument, we also see that $u_{\textup{rem}, n} \in L^\infty_t L^2_x$, as required.

Thus, it only remains to show that $u_{\textup{lin}}$, $\ulow$, $(\Jfree(t) u(t) - u_{\textup{lin}}(t))$ and $u_{\text{rem}, n}$ satisfy~\eqref{eqn:decay-lemma-hypo}.  For $u_{\textup{rem}, n}$, this again follows from the inductive assumption that $\lVert \partial_x u_{\textup{rem}, n}(t) \rVert_{L^2} \to 0$ together with the Gagliardo-Nirenberg inequality, and similar arguments apply to $\ulow$ and $(\Jfree(t) u(t) - u_{\textup{lin}}(t))$.  For $u_{\textup{lin}}$, the required decay follows immediately from the RAGE theorem.

\bibliography{sources}
\bibliographystyle{plain}

\end{document}